\documentclass[12pt]{amsart}
\usepackage[mathcal]{eucal}
\usepackage{amsmath, amssymb, graphicx, hyperref, labelfig
}
\usepackage[all]{xy}

\newtheorem{thm}{Theorem}
\newtheorem{prop}[thm]{Proposition}
\newtheorem{lem}[thm]{Lemma}
\newtheorem{sublem}[thm]{Sublemma}
\newtheorem{cor}[thm]{Corollary}

\newtheorem{con}[thm]{Conjecture}

\theoremstyle{remark}
\newtheorem{rem}[thm]{Remark}

\theoremstyle{definition}

\newcommand{\C}{\mathbb C}
\newcommand{\R}{\mathbb R}
\newcommand{\Z}{\mathbb Z}

\newcommand{\End}{\mathrm{End}}

\newcommand{\E}{\mathrm{e}}
\newcommand{\I}{\mathrm{i}}

\newcommand{\QDL}{{\mathrm{QDL}}}
\newcommand{\li}{\operatorname{\mathrm{li}_2}}
\newcommand{\lih}{\operatorname{\mathrm{li}_2^\hbar}}
\newcommand{\Lih}{\operatorname{\mathrm{Li}_2^\hbar}}
\newcommand{\lin}{\operatorname{\mathrm{li}_2^{\frac2n}}}
\newcommand{\Lin}{\operatorname{\mathrm{Li}_2^{\frac2n}}}
\newcommand{\Tr}{\operatorname{\mathrm{Trace}}}
\newcommand{\vol}{\operatorname{\mathrm{vol}}}
\newcommand{\SL}{\mathrm{SL}_2(\C)}
\newcommand{\PSL}{\mathrm{PSL}_2(\C)}
\newcommand{\SSS}{\mathcal{K}^q}
\newcommand{\XX}{\mathcal X_{\SL}}
\newcommand{\XP}{\mathcal X_{\PSL}}

\newcommand{\dbar}{/\negthinspace\negthinspace/}
\newcommand{\hyp}{\mathrm{hyp}}
\newcommand{\vbar}{\,|\,}

\renewcommand{\leq}{\leqslant}
\renewcommand{\geq}{\geqslant}
\renewcommand{\phi}{\varphi}
\renewcommand{\epsilon}{\varepsilon}

\renewcommand{\Re}{\operatorname{\mathfrak{Re}}}
\renewcommand{\Im}{\operatorname{\mathfrak{Im}}}

\title
[Asymptotics of quantum invariants of surface diffeomorphisms II]
{Asymptotics of quantum invariants\\ of surface diffeomorphisms II:
\\
The figure-eight knot complement}

\author{Francis Bonahon}
\address {Francis Bonahon, Department
of Mathematics,  University of
Southern California, Los Angeles
CA~90089-2532, U.S.A.}
\email{fbonahon@usc.edu}
\urladdr{https://dornsife.usc.edu/francis-bonahon/}

\address {Francis Bonahon, Department of Mathematics,  Michigan State University, 
East Lansing MI 48824, U.S.A.}
\email{bonahonf@msu.edu}

\author{Helen Wong}
\address {Helen Wong, Department of Mathematics,  Claremont McKenna College, 
Claremont CA 91711, U.S.A.}
\email{hwong@cmc.edu}
\urladdr{https://sites.google.com/view/helenwong/}

\author{Tian Yang}
\address {Tian Yang, Department of Mathematics,  Texas A\&M University,
College Station TX 77843, U.S.A.}
\email{tianyang@math.tamu.edu}
\urladdr{https://www.math.tamu.edu/~tianyang/}

\date{\today}

\thanks{This work was partially supported by the grants DMS-1711297, DMS-2005656 (PI: Francis Bonahon), DMS-1841221, DMS-1906323 (PI: Helen Wong) and DMS-1812008 (PI: Tian Yang) from the US National Science Foundation, as well as a Birman Fellowship from the American Mathematical Society and a Simons Fellowship from the Simons Foundation (PI: Helen Wong).}

\begin{document}
\maketitle

\begin{abstract}
In earlier work, the authors introduced a conjecture which, for an orientation-preserving diffeomorphism $\phi \colon S \to S$ of a surface, connects a certain quantum invariant of $\phi$ with the hyperbolic volume of its mapping torus $M_\phi$. This article provides a proof of this conjecture in the simplest case where it applies, namely when the surface $S$ is the one-puncture torus and  the mapping torus $M_\phi$ is the complement of the figure-eight knot. 
\end{abstract}

This article is the second installment in a series of three, starting with \cite{BWY1} and ending with \cite{BWY3}. In \cite{BWY1}, we introduced a conjecture that connects two apparently unrelated quantities associated to an orientation-preserving diffeomorphism $\phi \colon S \to S$ of a surface $S$. One is a quantum invariant based on the representation theory of the Kauffman bracket skein algebra $\SSS(S)$, and consists of a linear isomorphism $\Lambda_{\phi, r}^q \colon V \to V$ of a large vector space $V$. This isomorphism $\Lambda_{\phi, r}^q$ is defined up to conjugation and up to multiplication by a scalar of modulus 1, and depends on the diffeomorphism $\phi$, on a character $[r] \in \XX(S)$  represented by a group homomorphism $r \colon \pi_1(S) \to \SL$ and  invariant under the action of $\phi$, on a root of unity $q$ and on certain complex puncture weights strongly constrained by the earlier data.  The other invariant comes from hyperbolic geometry, and is the volume $ \vol_\hyp(M_\phi)$ of the complete hyperbolic metric of the mapping torus $M_\phi$, obtained from the product $S\times [0,1]$ by gluing $S\times \{0\}$ to $S\times \{1\}$ through $\phi$. 

The conjecture of \cite{BWY1} is essentially that, if we choose the root of unity $q_n= \E^{\frac{2\pi\I}n}$, the modulus of the trace of $\Lambda_{\phi, r}^{q_n}$ exponentially grows as $\E^{\frac n{4\pi} \vol_\hyp(M_\phi)}$ as $n$ tends to $\infty$. See \S \ref{sect:VolConf} for a precise statement of this conjecture. 

The current article proves the conjecture for what is essentially the simplest nontrivial example, where the surface $S$ is the one-puncture torus $S_{1,1}$ and where, using the classical identification of the mapping class group $\pi_0\, \mathrm{Diff}^+(S_{1,1}) $ with $ \mathrm{SL}_2(\Z)$, $\phi$ corresponds to the matrix $LR= \left( 
\begin{smallmatrix}
 2&1\\1&1
\end{smallmatrix}
  \right)$. This example  is also famous, because the mapping torus $M_\phi$ is then diffeomorphic to the complement of the figure-eight knot in $S^3= \R^3 \cup \{\infty\}$. It has the advantage that both its hyperbolic geometry and its quantum topology are relatively simple. On the hyperbolic side, the hyperbolic metric of $M_\phi$ is isometric to the union of two regular hyperbolic ideal tetrahedra, and its volume is
$$
 \vol_\hyp(M_{LR}) = 6\Lambda\left(\textstyle  \frac\pi3 \right) \approx 2.02988... .
$$
where $\Lambda(\theta)$ denotes the Lobachevsky function. On the quantum topology side, the trace of $\Lambda_{\phi, r}^{q_n}$ can be factored as a product of two simpler terms and, more critically, a cancellation of phases (see Remark~\ref{rem:NoPhase}) makes the asymptotic analysis accessible by elementary methods. 

The techniques used in this article are completely elementary, although sometimes challenging. These are intended to serve as a ``sanity check'' for the more sophisticated methods developed in \cite{BWY3}, which combine harmonic analysis and complex geometry  to address more general diffeomorphisms of the one-puncture torus. In particular, we will already encounter here, in a very concrete way, some of the cancellations that occur in the more general framework. 

A large portion of this article, namely \S \ref{sect:AsymptoticsDq}, is also used in a critical way in \cite{BWY3}. Indeed, the  limit of a certain normalization factor $\left| D^q(u) \right|^{\frac1n}$ is determined in that section. The mere existence of such a finite limit is already quite surprising. 

\section{The Volume Conjecture for surface diffeomorphisms}
\label{sect:VolConf}

We briefly describe the conjecture that motivates this article, referring to \cite{BWY1} for details. 

Let $\phi \colon S \to S$ be an orientation-preserving diffeomorphism of an oriented surface $S$ of finite topological type. We assume that $\phi$ is pseudo-Anosov so that its \emph{mapping torus} $M_\phi$, obtained from the product $S\times [0,1]$ by gluing $S\times \{0\}$ to $S\times \{1\}$ through $\phi$, admits a complete hyperbolic metric. 

The diffeomorphism $\phi$ acts on any object that is naturally associated to the surface $S$. In particular, it acts on the \emph{$\SL$--character variety}
$$
\XX(S) = \{ \text {group homomorphisms } r \colon \pi_1(S) \to \SL \} \dbar \SL
$$ 
where $\SL$ acts on group homomorphisms $r \colon \pi_1(S) \to \SL$ by conjugation and where the quotient is taken in the sense of geometric invariant theory. Let $\phi^* \colon \XX(S) \to \XX(S)$ denote the corresponding algebraic  isomorphism induced by $\phi$. 

The action of $\phi$ on $\XX(S)$ has many fixed points. Classical ones come from the monodromy $\bar r_\hyp \colon \pi_1(M_\phi) \to \PSL$ of the hyperbolic metric of the mapping torus $M_\phi$. However, there are many more fixed points when $S$ has at least one puncture. In fact, the dimension of the fixed point set of $\phi^* \colon \XX(S) \to \XX(S)$ has complex dimension $c$ near these $\phi$--invariant hyperbolic characters, where $c$ is the cardinality of the orbit space of the action of $\phi$ on the set of punctures of $S$. (See for instance \cite[\S 3.2]{BWY1}.)

Now, suppose that we are given:
\begin{enumerate}
 \item the diffeomorphism $\phi \colon S \to S$;
 \item a character $[r]$ that is in the smooth part of the character variety $\XX(S)$, and is fixed by the action of $\phi$;
 \item for each puncture $v$ of the surface $S$, a number $\theta_v\in \C$ such that 
\begin{enumerate}
 \item if $\alpha_v \in \pi_1(S)$ is represented by a loop going once around the puncture $v$, its image $r(\alpha_v) \in \SL$ has eigenvalues $-\E^{\pm\theta_v}$;
 \item the puncture weights $\theta_v$ are invariant under the action of $\phi$, in the sense that $\theta_{\phi(v)} = \theta_v$ for every puncture $v$; 
\end{enumerate}
\item a primitive $n$--root of unity $q$ with $n$ odd. 
\end{enumerate}

Deep results of \cite{BonWon3, FroKanLe1, GanJorSaf} associate to this data (where the puncture weights $\theta_v$ are replaced by $p_v= \E^{\frac1n \theta_v} + \E^{-\frac1n \theta_v}$) an irreducible representation $\rho \colon \SSS(S) \to \End(V)$ of the Kauffman bracket skein algebra $\SSS(S)$ which, up to isomorphism, is fixed under the action of $\phi$. This means that the representations $\rho$ and $\rho \circ \phi^* \colon \SSS(S) \to \End(V)$ are isomorphic, by a linear  isomorphism $\Lambda_{\phi, r}^q \colon V \to V$. We normalize $\Lambda_{\phi, r}^q $ so that its determinant has modulus 1. Then, the modulus $ \left| \Tr \Lambda_{\phi, r}^{q} \right| $ of its trace is uniquely determined; see \cite[Prop.~4]{BWY1}. 

\begin{con}
\label{con:Conjecture}
 Let the pseudo-Anosov surface diffeomorphism $\phi \colon S \to S$, the $\phi$--invariant smooth character $[r] \in \XX(S)$ and the $\phi$--invariants puncture weights $\theta_v$ as above be given.  For every odd $n$, consider the primitive root of unity $q_n = \E^{\frac{2\pi\I}n}$. Then
 $$
 \lim_{n \,\text{odd} \,\to \infty} \frac1n \log \left| \Tr \Lambda_{\phi, r}^{q_n} \right| = \frac1{4\pi} \vol_\hyp (M_\phi), 
 $$
 where $ \vol_\hyp (M_\phi)$ is the volume of the complete hyperbolic metric of the mapping torus $M_\phi$. 
\end{con}

The article is devoted to a proof of this conjecture for the simplest example where it applies, namely for a specific diffeomorphism of the one-puncture torus $S_{1,1}$; see Corollary~\ref{cor:VolConfForLR}. For that case, we will actually  prove in Theorem~\ref{thm:MainTheorem} a result that is stronger than the conjecture, by identifying constants $K_1$ and $K_3>0$ such that 
\begin{align*}
 \lim_{\substack{n\to \infty \\ n=1\, \mathrm{mod}\, 4}}  \left| \Tr \Lambda_{\phi, r}^{q_n} \right| \E^{-\frac n{4\pi} \vol_\hyp(M_\phi)} &= K_1
 \\
 \text{and }
  \lim_{\substack{n\to \infty \\ n=3\, \mathrm{mod}\, 4}}  \left| \Tr \Lambda_{\phi, r}^{q_n} \right| \E^{-\frac n{4\pi} \vol_\hyp(M_\phi)} &= K_3.
\end{align*}

This result is consistent with the dual convergence mode that was observed in the numerical experiments  of \cite[\S 2]{BWY1}. 

\section{The trace of the intertwiner for $\phi = LR$} 
\label{sect:ComputeTrace}

For the one-puncture torus $S_{1,1}$, the mapping class group $\pi_0 \,\mathrm{Diff}^+(S_{1,1})$ is determined by its action on the homology group $H_1(S_{1,1}) \cong \Z^2$, which provides an identification $\pi_0 \mathrm{Diff}^+(S_{1,1}) \cong \mathrm{SL}_2(\Z)$. A well-known property of $ \mathrm{SL}_2(\Z)$ is that every conjugacy class is represented by an element $\phi=\pm\phi_1 \phi_2 \dots \phi_k$ where each $\phi_i$ is equal to $L=\left(
\begin{smallmatrix}
 1&1\\0&1
\end{smallmatrix}
 \right)$ 
 or
 $R=\left(
\begin{smallmatrix}
 1&0\\1&1
\end{smallmatrix}
 \right)$. It turns out that the intertwiner $\Lambda_{\phi, r}^q$ and the volume $ \vol_\hyp (M_\phi)$  are both independent of the sign $\pm$. Also, there needs to be at least one $\phi_i$ equal to $L$ and one $\phi_j$ equal to $R$ if we want $\phi$ to be pseudo-Anosov. The simplest example where Conjecture~\ref{con:Conjecture} applies is therefore that of $\phi = LR$.

We apply the algorithm of \cite[\S 4]{BWY1} to determine the trace of the intertwiner $\Lambda_{LR, r}^q$ occurring in Conjecture~\ref{con:Conjecture}. 

We start with the ideal triangulation sweep associated to $\phi=LR$ and with a periodic edge weight system for this sweep, as in \cite[\S3.2]{BWY1}. In practice, this edge weight system is  a finite sequence $(a_0, b_0, c_0)$, $(a_1, b_1, c_1)$, $(a_2, b_2, c_2) \in \big( \C^* \big)^3$ such that
\begin{align*}
 a_1 &= b_0^{-1}
 &
 b_1 &= (1+ b_0)^2 a_0
 &
 c_1&= (1 + b_0)^{-2} b_0^2 c_0
 \\
 a_2 &= c_1^{-1}
 &
 b_2 &= (1+c_1)^2 b_1
 &
 c_2 &= (1+c_1)^{-2} c_1^2 a_1
 \\
 a_0 &= a_2
 & 
 b_0 &= b_2
 &
 c_0 &= c_2.
\end{align*}

This case is simple enough that these equations can be explicitly solved through the quadratic formula. Their solutions are of the form
\begin{align*}
 a_0 &= \frac{ -1 -\frac32 b_0 - b_2^{2} \pm \sqrt{-b_0^{3} - \frac74 b_0^{2} - b_0^{}}} {(1+b_0)^2}
 &
 c_0&= \frac{ -1 -\frac32 b_0 - b_2^2 \mp \sqrt{-b_0^3 - \frac74 b_0^{2} - b_0}}{b_0^2}
 \\
 a_1 &= b_0^{-1}
 &
 b_1 &= \textstyle  -1 -\frac32 b_0 - b_2^2 \mp \sqrt{-b_0^3 - \frac74 b_0^{2} - b_0}
 \\
 c_1 &=  \frac{ -1 -\frac32 b_0 - b_2^{2} \mp \sqrt{-b_0^{3} - \frac74 b_0^{2} - b_0^{}}} {(1+b_0)^2}
\end{align*}
as $b_0$ ranges over $\C-\{0,-1\}$ and $\pm \in \{+, -\}$ (with $\mp = - \pm$ and an arbitrary choice of complex square roots). 

As in  \cite[\S3.2]{BWY1}, such a periodic edge weight system defines a $\phi$--invariant character $[\bar r ] \in  \XP(S_{1,1})$. Conversely, there is a Zariski-open dense subset $U \subset \XP(S_{1,1})$ such that every  $\phi$--invariant character $[\bar r ] \in  U$ is obtained in this way; see for instance \cite[Lem. 10]{BWY1}. 

In particular, the periodic edge weight system  where $(a_0, b_0, c_0) = \left( \E^{-\frac{2\pi\I}3},  \E^{\frac{2\pi\I}3}, 0 \right)$ induces the character $[\bar r_\hyp] \in \XP(S_{1,1})$ coming from the monodromy of the complete hyperbolic metric of the mapping torus $M_\phi$; see for instance \cite[Chap. 8]{BonBook}. 

In the framework of Conjecture~\ref{con:Conjecture}, we are also given a puncture weight $\theta_v \in \C$ such that $\E^{\theta_v} = a_0b_0c_0$. 
Then, as in \cite[\S 4.7]{BWY1}, we choose ``logarithms'' $A_k$, $B_k$, $C_k$, $V_k \in \C^*$ such that
\begin{align*}
 \E^{A_k} &= a_k
&
 \E^{B_k} &= b_k
&
 \E^{C_k} &= c_k
&
 \E^{V_k} &= 1+a_k^{-1} 
\end{align*}
for every $k=0$, $1$, $2$. These logarithms are constrained by the properties that $$A_0 + B_0 + C_0 = \theta_v$$ and
\begin{align*}
A_1 &=  -B_0
&
A_2 &=  -C_1
\\
B_1&=  2 V_1 + A_0
&
B_2 &=  2 V_2 + B_1
\\
C_1 &= -2 V_1  +2B_0 + C_0
&
C_2 &= -2 V_2 +2 C_1 +A_1.
\end{align*}

Once these numbers are chosen, there exist integers $\widehat l_0$, $\widehat m_0$, $\widehat n_0 \in \Z$ such that
\begin{align*}
 A_0 &= A_2 + 2 \pi \I \widehat l_0
 &
 B_0 &= B_2 + 2 \pi \I \widehat m_0
 &
 C_0 &= C_2 + 2 \pi \I \widehat n_0.
\end{align*}
In addition, $\widehat l_0 + \widehat m_0 + \widehat n_0 =0$.

The following is the immediate application of  \cite[Prop. 32]{BWY1} to the case at hand.

\begin{prop}
\label{prop:FormulaTrace}
Consider the diffeomorphism $\phi=LR=\left(
\begin{smallmatrix}
 2&1\\1&1
\end{smallmatrix}
 \right)$ of the one-puncture torus $S_{1,1}$. 
Let $[r]\in \XX(S_{1,1})$ be the $\phi$--invariant character associated to the periodic edge weight system $(a_0, b_0, c_0)$, $(a_1, b_1, c_1)$, $(a_2, b_2, c_2)=(a_0, b_0, c_0)$, and let  $\theta_v\in \C$ be such that $\E^{\theta_v}= a_0 b_0 c_0$. For an odd integer $n$,  let $\Lambda_{\phi, r}^q$ be the intertwiner to the data of $\phi$, $[r]$, $q=\E^{\frac{2\pi\I}n}$ and $p_v =  \E^{\frac1n \theta_v}  +\E^{-\frac1n \theta_v} $, as in Conjecture~{\upshape \ref{con:Conjecture}}. 

Then, up to multiplication by a scalar with modulus $1$, the trace of $\Lambda_{\phi, r}^q$ is equal to
\begin{align*}
\Tr \Lambda_{\phi, r}^q
&=\frac1
{n  \left| D^q \big( q \E^{-\frac1n {A_1}} \big) \right|^{\frac1n}\left| D^q \big( q \E^{-\frac1n{A_2}} \big) \right|^{\frac1n} }
\\
&\qquad\qquad
\sum_{i_1,\, i_2 =1}^n
  \QDL^q \big( q \E^{-\frac1n{A_1}},  \E^{\frac1n V_1} \vbar  2i_1 \big)   \QDL^q \big( q \E^{-\frac1n{A_2}},  \E^{\frac1n V_2}  \vbar  2i_2 \big) \\
&\qquad\qquad \qquad\qquad \qquad\qquad
q^{ 2 i_1^2 + 2 i_2^2 - \widehat l_0 i_1  + \frac{ \widehat l_0 -\widehat m_0 +\widehat n_0}2 i_2 }
\end{align*}
where the quantities $A_1$, $A_2$, $V_1$, $V_2 \in \C$ and the integers $\widehat l_0$, $\widehat m_0$, $\widehat n_0\in \Z$ are defined as above and where, for $u$, $v\in \C$ with $v^n = 1+u^n$, 
$$
 \QDL^q(u,v \vbar  i) = v^{-i} \prod_{j=1}^i (1+ u q^{-2j})
$$
and \pushQED{\qed}
\begin{equation*}
D^q(u)  =  \prod_{i=1}^n \QDL^q(u, v \vbar  i)  = (1+u^n)^{- \frac{n+1}2} \prod_{j=1}^{n} (1+ u q^{-2j})^{n-j+1}.
\qedhere
\end{equation*} 
\end{prop}

A simplification that is very specific to this case is that the above double sum can be factored as the product of two single sums. More precisely, using the fact that $\widehat l_0 + \widehat m_0 + \widehat n_0 =0$ for a further simplification, 
\begin{align}
\notag
\Tr \Lambda_{\phi, r}^q
&=\frac1
{n  \left| D^q \big( q \E^{-\frac1n {A_1}} \big) \right|^{\frac1n}\left| D^q \big( q \E^{-\frac1n{A_2}} \big) \right|^{\frac1n} }
\notag
\\
\label{eqn:FormulaTrace}
&\qquad\qquad
\left( \sum_{i_1=1}^n
  \QDL^q \big( q \E^{-\frac1n{A_1}},  \E^{\frac1n V_1} \vbar  2i_1 \big)  q^{ 2 i_1^2 - \widehat l_0 i_1  }  \right)\\
&\qquad\qquad \qquad\qquad \qquad\qquad
\left( \sum_{i_2=1}^n
  \QDL^q \big( q \E^{-\frac1n{A_2}},  \E^{\frac1n V_2} \vbar  2i_2 \big)  q^{  2 i_2^2 - \widehat m_0 i_2 } \right).
  \notag
\end{align}

This reduces the question of finding the asymptotics of $\Tr \Lambda_{\phi, r}^q$ to the asymptotic analysis of sums
$$
\Sigma_n = \sum_{i=1}^n \QDL^q (u,v \vbar 2i) q^{2i^2 - \widehat k i}
$$
and  $n$--roots 
$$
\left| D^q (u) \right|^{\frac1n}
$$
with $q= \E^{\frac{2\pi\I}n}$, $u=\E^{\frac1n U}$, $v=\E^{\frac1n V}$ for fixed numbers $U$, $V\in \C$ such that $\E^V= 1+ \E^U$, and $\widehat  k\in \Z$ fixed. Note that the parameters $u_k= q \E^{-\frac1n A_k}$ occurring in  (\ref{eqn:FormulaTrace})  are clearly of the above type, since $u_k = \E^{\frac1n U_k}$ for  $U_k=2\pi\I - A_k$. 

\section{Asymptotics of the sum $\Sigma_n$} 
\label{sect:AsymptoticsSn}

This section is devoted to the asymptotic analysis, as $n$ odd tends to infinity, of the sum
\begin{equation}
\label{eqn:TheSumSn}
\Sigma_n = \sum_{i=1}^n \QDL^q (u,v \vbar 2i) q^{2i^2 - \widehat k i}
\end{equation}
where $q= \E^{\frac{2\pi\I}n}$, $u=\E^{\frac1n U}$ and $v=\E^{\frac1n V}$ for fixed numbers $U$, $V\in \C$ such that $\E^V= 1+ \E^U$, and $\widehat  k\in \Z$. 

The arbitrarily chosen example of Figure~\ref{fig:SingleSum}, where $U \approx - 2.58581 + 6.05389\, \I$, $\widehat k=5$ and   $n=4001$, may be a good illustration of the phenomena that occur in the proof. Each dot represents, in the complex plane $\C$, one term of the sum $\Sigma_n$. Most of the terms are concentrated near the origin (and hardly visible on the picture), and do not contribute very much to the sum. The  dots corresponding to the largest terms are arranged in three large ``petals'', whose size will be shown to exponentially increase with $n$. Two of these petals are opposite each other, and we will see that their contributions to the sum relatively cancel out. More precisely, the sum of the corresponding terms is small compared to the leading contribution, which comes from the dots in the third petal.

\begin{figure}[htbp]
\includegraphics[width=.5\textwidth]{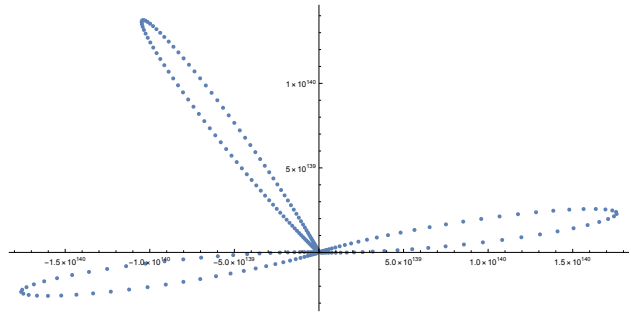}
\caption{The terms of the sum $\Sigma_n$, for a specific example.}
\label{fig:SingleSum}
\end{figure}

\subsection{The discrete and continuous quantum dilogarithms}
\label{sect:QuantumDilogEstimates}

The goal of this section is to obtain estimates on the size, as well as the phase, of the terms $ \QDL^q (u,v \vbar 2i) q^{2i^2 - \widehat k i}$ occurring in the sum $\Sigma_n$ of (\ref{eqn:TheSumSn}). 

For a primitive $n$--root of unity $q$ with $n$ odd, and for $u$, $v\in \C$ with $v^n = 1+u^n$, the term 
$$
 \QDL^q(u,v \vbar  i) = v^{-i} \prod_{j=1}^i (1+ u q^{-2j}).
$$
is the Faddeev-Kashaev  \emph{discrete quantum dilogarithm function}. This is an $n$--periodic function of the integer $i$, which enables one to define it for negative $i\in \Z$ as well.

A well-known interpolation of this discrete function is provided by the  continuous quantum dilogarithm function of Faddeev \cite{Fadd, FadKash}. Because many different normalizations are used in the literature, we first summarize our notation and conventions, as well as the basic properties that we will need. We refer to \cite[\S 4.1]{CheFoc1} or  \cite[\S 2.3]{WongYang1}, for instance,  for proofs of these properties.

We will actually distinguish two continuous quantum dilogarithm functions, a ``small'' one and a ``big'' one. 

For $\hbar>0$, the \emph{small continuous quantum dilogarithm} is the function
\begin{equation}
\label{eqn:DefQuantumDilogIntegral}
\lih(z) ={2\pi \I \hbar} \int_\Omega  \frac{\E^{(2z-\pi)t}}{4t\sinh(\pi t) \sinh(\pi \hbar t)} \,dt ,
\end{equation}
defined for $-\frac{\pi\hbar}2< \Re z < \pi+ \frac{\pi\hbar}2$, 
where the integration domain $\Omega$ is obtained from $\R$ (oriented from $-\infty$ to $+\infty$) by replacing a small interval around 0 with a semi-circle lying above the real line, in order to bypass the pole of the integrand at 0. 

The following property shows that this function can be seen as a deformation of the \emph{classical dilogarithm}
$$
\li(u) = - \int_0^u \frac{\log(1-t)}t\, dt
$$
defined over the open simply-connected subset $ \C- \left[1, +\infty\right[$. 

\begin{prop}
\label{prop:QuantumNonquantumDilog}
For every $z$ with $0< \Re z < \pi$ (so that $\E^{2\I z} \in \C- \left[1, +\infty\right[$),
$$ \lih (z) = \li(\E^{2\I z}) + O(\hbar^2)$$
as $\hbar\to0$, and this uniformly on compact subsets of the strip $\{z \in \C; 0 < \Re z < \pi \}$. 
\end{prop}
\begin{proof}
 See for instance \cite[\S 4.1]{CheFoc1} or \cite[\S 2.3]{WongYang1}. 
\end{proof}

The \emph{big continuous quantum dilogarithm}  is the function
$$
\Lih(z) = \exp \left(  {\textstyle\frac1{ 2\pi\I\hbar} } \lih(z) \right)
$$
also defined for $-\frac{\pi\hbar}2< \Re z < \pi+ \frac{\pi\hbar}2$. 

\begin{prop}
\label{prop:BigQuantumDilog}
The function 
$
\Lih(z)
$
uniquely extends to a meromorphic function on the whole complex plane, whose poles are the points  $\pi + \frac{\pi\hbar}2 + a\pi + b\pi\hbar$ for all integers $a$, $b\geq 0$, and whose zeros are the points $-\frac{\pi\hbar}2 - a\pi - b\pi\hbar$ for all integers $a$, $b\geq 0$.

This extension satisfies the functional equations
\begin{align}
\label{eqn:BigQuantumDilogPiHbarPeriodic}
\Lih (z+ \pi \hbar) &= (1- \E^{2 \I z +\pi\I \hbar} )^{-1} \Lih (z )\\
\label{eqn:BigQuantumDilogPiPeriodic}
\Lih (z+\pi) &= (1+ \E^{\frac{2\I z}\hbar})^{-1} \Lih (z ). 
\end{align}
\end{prop}
\begin{proof}
 Again, see for instance \cite[\S 4.1]{CheFoc1} or \cite[\S 2.3]{WongYang1}. 
\end{proof}

Setting $\hbar = \frac 2n$, Equation~(\ref{eqn:BigQuantumDilogPiHbarPeriodic}) easily provides the following connection between the discrete and continuous quantum dilogarithms. 

\begin{cor}
\label{cor:DiscreteContinuousQuantumDilogPeriodic}
\pushQED{\qed}
Let $U$ and $V\in \C$ such that $\E^V =1+\E^U$ be fixed. For every odd $n$, set $q = \E^{\frac{2\pi\I}n}$, $u = \E^{\frac1n U}$ and $v=\E^{\frac1n V}$. Then, for every $j\in \Z$,
\begin{equation*}
\QDL^q (u, v \vbar j)=
\E^{-\frac jn V}\ 
\frac
{\Lin \left( \frac\pi2 - \frac\pi n +  \frac U{2n\I}   - {\frac{2\pi j}{n}} \right) }
{\Lin \left(   \frac\pi2 -  \frac\pi n +  \frac U{2n\I}  \right) }.
\qedhere
\end{equation*}
\end{cor}

This connection with the continuous quantum dilogarithm will enable us to obtain estimates for the discrete quantum logarithmic function $\QDL^q (u, v \vbar j)$. These involve the \emph{Lobachevsky function}, which is the function
$$
\Lambda( \theta) = - \int_0^\theta \log \left| 2 \sin t \right| \, dt
$$
and is classically related to the dilogarithm function $\li(z)$ by the property that
\begin{equation}
 \label{eqn:LobaAndDilog}
 \li\left( \E^{2\I \theta} \right) = {\textstyle\frac{\pi^2}6} -\theta(\pi-\theta) + 2\I \Lambda(\theta)
\end{equation}
for every $\theta \in [0,\pi]$. 

This analysis of $\QDL^q (u, v \vbar j)$ splits into two cases, according to whether the numbers $\frac{2\pi j}n$ corresponding to the indices $j$ are close to $\frac\pi2$ modulo $\pi$ or not. 
We first consider the latter case, where $\frac{2\pi j}n$ stays away from the points $\frac\pi2$ modulo $\pi$. By $n$--periodicity of the function $\QDL^q (u, v \vbar j)$, we can restrict attention to the case where $-\frac\pi2 \leq \frac{2\pi j}n \leq \frac{3\pi}2$. 

\begin{lem}
\label{lem:EstimateQdlFarFromPi/2}
 Fix numbers $U$, $V\in \C$ with $\E^V = 1+ \E^U$, and $\delta>0$. For every odd $n$, let $q=\E^{\frac{2\pi \I}n}$, $u = \E^{\frac1n U}$ and $v = \E^{\frac1n V}$. The following estimates hold for $n$ sufficiently large and for every integer $j \in \Z$.
\begin{enumerate}
 \item If $-\frac\pi2 + \delta \leq \frac{2\pi j}n \leq \frac\pi2 - \delta$, 
 $$
\QDL^q (u, v \vbar j) =  \E^{n \widetilde f \left( \frac{2\pi j}n \right)}\  \widetilde g_n \left( \textstyle\frac{2\pi j}n \right)
 $$
 where $\widetilde f$, $\widetilde g_n \colon \left[ -\frac\pi2+ \delta, \frac\pi2 - \delta \right] \to \C$ are continuous functions such that
\begin{align*}
\widetilde f(t) &= \textstyle \frac 1{2\pi}   \Lambda \left( \textstyle \frac\pi2 - t \right) - \frac1{4\pi} t^2 \I
\\
\widetilde g_n(t) &=    2^{\frac{U-2\pi\I}{4\pi\I}}\ 
 \left(1+ \E^{- 2t \I} \right)^{\frac{2\pi\I - U}{4\pi\I}} 
 \E^{ - \frac {t V}{2\pi}} 
\E^{ O \left( \frac1{n} \right)}.
\end{align*}
 \item If $\frac\pi2 + \delta \leq \frac{2\pi j}n \leq \frac{3\pi}2 - \delta$, 
 $$
\QDL^q (u, v \vbar j) = (-1)^j  \E^{n \widehat f \left( \frac{2\pi j}n \right)}\  \widehat g_n \left( \textstyle\frac{2\pi j}n \right)
 $$
 where $\widehat f$, $\widehat g_n \colon \left[ -\frac\pi2+ \delta, \frac\pi2 - \delta \right] \to \C$ are continuous functions such that
\begin{align*}
\widehat f(t) &=   \textstyle \frac 1{2\pi}    \Lambda \left(  \frac{3\pi}2 - t \right)- \frac1{4\pi} t^2 \I
\\
\widehat g_n(t) &=2^{\frac{U-2\pi\I}{4\pi\I}}\ 
 \E^{-\frac{n\pi\I}4}\  
\left( 1 - \I^n \E^{ \frac 12 U }  \right)
 \left(1+ \E^{- 2t \I} \right)^{\frac{2\pi\I - U}{4\pi\I}} 
 \E^{ - \frac {t V}{2\pi}}
\E^{ O \left( \frac1{n} \right)} .
\end{align*}
\end{enumerate}
In both cases, the constants hidden in the condition ``for $n$ sufficiently large'' and in the Landau symbol $O(\ )$ can be chosen to depend only on $U$, $V$ and $\delta$.

\end{lem}
\begin{proof}
We will rely on Corollary~\ref{cor:DiscreteContinuousQuantumDilogPeriodic} and Proposition~\ref{prop:QuantumNonquantumDilog}. 

 First consider the case where $-\frac\pi2 +\delta \leq \frac {2\pi j} n \leq \frac\pi2 - \delta $. If $t \in  \left[ -\frac\pi2 +\delta , \frac\pi2 - \delta  \right] $, then 
$$
\textstyle
{\textstyle\frac12} \delta \leq \Re \left( \frac \pi2 -\frac\pi n + \frac U{2n\I} -  t \right)  \leq \pi - {\textstyle\frac12}\delta 
$$
for $n$ large enough. 
Because this quantity is in a compact subset of $\left]0,\pi\right[$, we can combine Proposition~\ref{prop:QuantumNonquantumDilog} with linear approximation to obtain
\begin{align*}
\lin\left( \textstyle \frac \pi2 -\frac\pi n + \frac U{2n\I} -t \right)
&= \li \big(- \E^{\frac{U-2\pi\I}n  - 2t \I } \big) + O\left(\textstyle \frac1{n^2}\right)
\\
&= \textstyle 
\li \left(- \E^{-2t \I } \right) +  \frac{U-2\pi\I}n  \frac{d}{ds}\li \left(- \E^{s- 2t \I } \right)_{|s=0}  + O( \frac1{n^2})\\
&= \textstyle 
\li \left(- \E^{-2t \I } \right)  -  \frac{U-2\pi\I}n  \log \left( 1+ \E^{-  2t\I } \right)  + O \left( \frac1{n^2} \right). 
\end{align*}
It follows that
\begin{align*}
\Lin\left( \textstyle \frac \pi2 -\frac\pi n + \frac U{2n\I} - t  \right)
&=
\exp \big( {\textstyle \frac{n}{4\pi\I}}
\lin \big( \textstyle \frac \pi2 -\frac\pi n + \frac U{2n\I} - t\big) \big)
\\
&=\exp \left( {\textstyle \frac{n}{4\pi\I}}  \li\left(- \E^{- 2t \I } \right) + {\textstyle\frac{2\pi\I -U}{4\pi \I}}  \log \left(1+ \E^{- 2t \I }  \right) + O \left( \textstyle\frac1{n} \right)
\right)
\\
&=
 \left(1+ \E^{- 2t \I } \right)^{\frac{2\pi\I -U}{4\pi\I}}
 \exp \left( {\textstyle \frac{n}{4\pi\I}}  \li\left(-\E^{- 2t \I } \right)
  + O \left( \textstyle\frac1{n} \right)
\right).
\end{align*}

We now use the  relation (\ref{eqn:LobaAndDilog})  connecting the dilogarithm and Lobachevsky functions. This property enables us to conclude that
\begin{equation}
\label{eqn:EstimateQDL1}
\begin{aligned}
\Lin\left( \textstyle \frac \pi2 -\frac\pi n + \frac U{2n\I} - t  \right)
&=
 \left(1+ \E^{- 2t \I} \right)^{\frac{2\pi\I - U}{4\pi\I}}
 \exp \left( \textstyle \frac{n}{2\pi}  \Lambda \left( \textstyle\frac\pi2- t\right)
 + \frac{n\pi}{48}\I -  \frac{n}{4\pi} t^2 \I
  + O \left( \textstyle\frac1{n} \right)
\right).
\end{aligned}
\end{equation}

In particular, the case $t=0$ gives
\begin{align}
\label{eqn:EstimateQDL2}
\Lin\left( \textstyle  \frac \pi2 -\frac\pi n + \frac U{2n\I}  \right)
&=
2^{\frac{2\pi\I - U}{4\pi\I}}
 \exp \left( \textstyle \frac{n\pi}{48} \I
  + O \left( \textstyle\frac1{n} \right)
\right)
\end{align}
since $\Lambda(\frac\pi2)=0$.

Combining Corollary~\ref{cor:DiscreteContinuousQuantumDilogPeriodic} with  (\ref{eqn:EstimateQDL1}) and  (\ref{eqn:EstimateQDL2}), we now obtain that for $t=\frac{2\pi j}n$
 \begin{align*}
\QDL^q (u, v \vbar j)
 &=   \E^{{\frac n{2\pi }}   \Lambda \left(  \frac\pi2 - t \right) - \frac n{4\pi}t^2 \I} \ 
 2^{\frac{U-2\pi\I}{4\pi\I}}\ 
 \left(1+ \E^{- 2t \I} \right)^{\frac{2\pi\I - U}{4\pi\I}} 
 \E^{ - \frac {t V}{2\pi}}\
\E^{ O \left( \frac1{n} \right)}
\end{align*}
for every $j\in \Z$ with $-\frac\pi2 +\delta \leq \frac {2\pi j} n \leq \frac\pi2 - \delta $. This is what we wanted. More precisely,
 $$
\QDL^q (u, v \vbar j) =  \E^{n  \widetilde f \left( \frac {2\pi j} n \right)}\  \widetilde g_n \left(\textstyle  \frac {2\pi j} n \right) 
 $$
 with  $ \widetilde f(t) =  \frac 1{2\pi}    \Lambda \left(  \frac\pi2 - t \right) - \frac1{4\pi} t^2 \I$ and
 $$
 \widetilde g_n(t)  =  2^{\frac{U-2\pi\I}{4\pi\I}}\ 
 \left(1+ \E^{- 2t \I} \right)^{\frac{2\pi\I - U}{4\pi\I}} 
 \E^{ - \frac {t V}{2\pi}} \
\E^{ O \left( \frac1{n} \right)}
 $$
 for every $t \in  \left[ -\frac\pi2 +\delta , \frac\pi2 - \delta  \right] $.
 In addition, if we backtrack through the argument that led us here, we see that the precise definition of $\widetilde g_n(t)$ is
 $$
 \widetilde g_n(t) =  
 \E^{ - \frac {t V}{2\pi}} \,
\frac
{\Lin\left(  \frac \pi2 -\frac\pi n + \frac U{2n\I} -  t \right)}
{\Lin\left( \frac \pi2 -\frac\pi n + \frac U{2n\I} \right)} 
 \E^{ -{\frac n{2\pi }}   \Lambda \left(  \frac\pi2 - t \right) + \frac n{4\pi}t^2 \I} ,
 $$
 which shows that this function is continuous, since $\Lin$ is continuous on the strip $\{z\in \C; 0<\Re z <\pi\}$.  

This concludes the case where $-\frac\pi2 +\delta \leq \frac {2\pi j} n \leq \frac\pi2 - \delta $.

\medskip
For the second case, where $\frac\pi2 +\delta \leq \frac {2\pi j} n \leq \frac{3\pi}2 - \delta $, the argument is very similar with a few small modifications. Indeed, if $t \in \left[ \frac\pi2 +\delta,  \frac{3\pi}2- \delta \right]$,
$$
\textstyle
- \pi+ {\textstyle\frac12} \delta \leq \Re \left(  \textstyle \frac \pi2 -\frac\pi n + \frac U{2n\I} -t \right) \leq - {\textstyle\frac12}\delta 
$$
for $n$ large enough, but  this quantity is not in the interval $\left]0,\pi\right[$ any more. 

We first use the relation (\ref{eqn:BigQuantumDilogPiPeriodic}) of Proposition~\ref{prop:BigQuantumDilog} to express
\begin{equation}
\label{eqn:EstimateQDL3}
\begin{aligned}
\Lin\left(  \textstyle \frac \pi2 -\frac\pi n + \frac U{2n\I} - \frac {2\pi j}n  \right)
&=
\left( 1 + \E^{n\I \left(   \frac \pi2 -\frac\pi n + \frac U{2n\I} - \frac {2\pi j}n  \right)}
\right)
\Lin\left(  \textstyle \frac {3\pi} 2 -\frac\pi n + \frac U{2n\I} - \frac {2\pi j}n  \right)
\\
&=
\big( 1 - \I^n \E^{ \frac 12 U }  \big)
\Lin\left( \textstyle \frac {3\pi} 2 -\frac\pi n + \frac U{2n\I} - \frac {2\pi j}n  \right)
\end{aligned}
\end{equation}
in terms of $ \frac {3\pi} 2 -\frac\pi n + \frac U{2n\I} - \frac {2\pi j}n  $, whose real part is now in a compact subset  of $\left]0,\pi\right[$ for $n$ sufficiently large. 

Then, if we replace $t $  by $t -\pi$ in the argument used to prove (\ref{eqn:EstimateQDL1}),
\begin{equation}
\label{eqn:EstimateQDL4}
\begin{aligned}
\Lin\left( \textstyle \frac {3\pi}2 -\frac\pi n + \frac U{2n\I} - t  \right)
&=
 \left(1+ \E^{- 2t \I} \right)^{\frac{2\pi\I - U}{4\pi\I}}
 \\
 &\qquad\qquad
 \exp \left( \textstyle \frac{n}{2\pi}  \Lambda \left( \textstyle\frac{3\pi}2- t \right) 
 - \frac{11n\pi}{48}\I - \frac{n}{4\pi}t^2 \I  + \frac n2  t \I 
  + O \left( \textstyle\frac1{n} \right)
\right)
\\
&=
 \E^{ \frac n2 t \I }  \left(1+ \E^{- 2t \I} \right)^{\frac{2\pi\I - U}{4\pi\I}}
 \\
 &\qquad\qquad
 \exp \left( \textstyle \frac{n}{2\pi}  \Lambda \left( \textstyle\frac{3\pi}2- t \right) 
 - \frac{11n\pi}{48}\I - \frac{n}{4\pi} t^2 \I 
  + O \left( \textstyle\frac1{n} \right)
\right)
\end{aligned}
\end{equation}
for every $t \in \left[ \frac\pi2 +\delta,  \frac{3\pi}2- \delta \right]$. 

The reason for singling out the term $ \E^{ \frac n2 t \I } $ in the second part of (\ref{eqn:EstimateQDL4}) is that it is equal to $(-1)^j$ when $t= \frac {2\pi j}n $ with $j\in \Z$. Then, combining Corollary~\ref{cor:DiscreteContinuousQuantumDilogPeriodic} with  (\ref{eqn:EstimateQDL2}), (\ref{eqn:EstimateQDL3}) and (\ref{eqn:EstimateQDL4})  gives, for every $j\in \Z$ with $\frac\pi2 +\delta< \frac {2\pi j}n <\frac{3\pi}2- \delta$, 
 \begin{align*}
\QDL^q (u, v \vbar j)
 &= (-1)^j  \ \E^{ {\frac n{2\pi }} \Lambda \left(  \frac{3\pi}2 - t \right) - \frac n{4\pi}t^2 \I} 
\\
&\qquad\qquad
2^{\frac{U-2\pi\I}{4\pi\I}}\ 
 \E^{-\frac{n\pi\I}4}\  
\big( 1 - \I^n \E^{ \frac 12 U }  \big)
 \left(1+ \E^{- 2t \I} \right)^{\frac{2\pi\I - U}{4\pi\I}} 
 \E^{ - \frac {t V}{2\pi}}
\E^{ O \left( \frac1{n} \right)}
\end{align*}
if $t= \frac{2\pi j}n$. 

This  expression is clearly of the form we wanted, namely
 $$
\QDL^q (u, v \vbar j)  = (-1)^j \    \E^{n  \widehat f \left( \frac {2\pi j} n \right)} \ \widehat g_n \left(\textstyle  \frac {2\pi j} n \right)
 $$
 with $\widehat  f(t) = \textstyle \frac 1{2\pi}    \Lambda \left(  \frac{3\pi}2 - t \right)- \frac1{4\pi} t^2 \I$ and 
 $$
 \widehat g_n(t) =
 2^{\frac{U-2\pi\I}{4\pi\I}}\ 
 \E^{-\frac{n\pi\I}4}\  
\big( 1 - \I^n \E^{ \frac 12 U }  \big)
 \left(1+ \E^{- 2t \I} \right)^{\frac{2\pi\I - U}{4\pi\I}} 
 \E^{ - \frac {t V}{2\pi}}\ 
\E^{ O \left( \frac1{n} \right)}
 $$
for every $t \in \left[ \frac\pi2 +\delta , \frac{3\pi}2 - \delta  \right]$.

We  need to be a little more careful about the continuity of the function $\widehat g_n$, because of our use of the two simplifications  that  $\E^{n\I\left( \frac\pi2 - \frac\pi n + \frac U{2n\I}-t\right)} = -\I^n \E^{\frac12 U}$ and $ \E^{ \frac n2 t \I } =(-1)^j$ when $t = \frac{2\pi j}n$. Tracking down the precise definition of $\widehat g_n(t)$ through the argument, we see that
$$
\widehat g_n(t) = \E^{-\frac n2 t\I}
 \E^{ - \frac {t V}{2\pi}}
\frac
{\big( 1 - \I^n \E^{ \frac 12 U }  \big) \Lin\left( \textstyle \frac {3\pi}2 -\frac\pi n + \frac U{2n\I} - t  \right)}
{\Lin\left( \textstyle \frac {\pi}2 -\frac\pi n + \frac U{2n\I}   \right)} \ 
 \E^{ -{\frac n{2\pi }}   \Lambda \left(  \frac{3\pi}2 - t \right) + \frac n{4\pi}t^2 \I} 
$$
for every $t \in \left[ \frac\pi2 +\delta , \frac{3\pi}2 - \delta  \right]$. It is now clear that  $\widehat g_n$ is continuous for $n$ large enough, since  $ \frac {3\pi}2 -\frac\pi n + \frac U{2n\I} - t  $ then belongs to the strip $\{z\in \C; 0<\Re z <\pi\}$ where $\Lin$ is continuous. 
\end{proof}

We are also going to need an estimate of $\QDL^q (u, v \vbar j)$ for the values of $j$ that are not addressed in Lemma~\ref{lem:EstimateQdlFarFromPi/2}. By the $n$--periodicity of $\QDL^q (u, v \vbar j)$, these remaining values of $j$ modulo $n$  correspond to $-\frac\pi2 -\delta \leq \frac{2\pi j}n \leq -\frac\pi2+ \delta$ or $\frac\pi2 -\delta \leq \frac{2\pi j}n \leq \frac\pi2+ \delta$. The upshot of our analysis will be that the corresponding values for $\QDL^q (u, v \vbar j)$ are relatively small.

\begin{lem}
\label{lem:EstimateQdlNearPi/2}
Given $U$, $V\in \C$ with $\E^V=1+ \E^U$, choose $q=\E^{\frac{2\pi \I}n}$, $u= \E^{\frac1n U}$ and $v= \E^{\frac1n V}$ for every $n$. If  $\delta>0$ is sufficiently small and if $n$ is sufficiently large,  then 
$$
\left\lvert  \QDL^q (u, v \vbar j) \right\rvert
=O\left(  \E^{\frac{n}{2\pi} \Lambda \left(2 \delta \right ) }  \right)
$$
whenever  $-\frac\pi2 -\delta \leq \frac{2\pi j}n \leq -\frac\pi2+ \delta$ or $\frac\pi2 -\delta \leq \frac{2\pi j}n \leq \frac\pi2+ \delta$. 

In addition, the constants hidden in the condition ``\,$n$ sufficiently large'' and in the Landau symbol $O(\ )$ can be chosen to depend only on $U$, $V$ and $\delta$. 
\end{lem}
The term $\Lambda(2\delta)$ in the estimate could be replaced by $\Lambda(\delta')$ for any $\delta' > \delta$. 
\begin{proof}
By definition of $\QDL^q (u, v \vbar j)$, 
\begin{align*}
 \frac{ \left\lvert  \QDL^q (u, v \vbar j+1)   \right\rvert }{\left\lvert  \QDL^q (u, v \vbar j)   \right\rvert }
 &=
 \frac{\left\lvert  1+ u q^{-2j-2}\right\rvert}{\left\lvert v \right\rvert}
 =  \frac{\big\lvert  1  + \E^{ \frac 1n U - \frac{4\pi j}n \I - \frac{4\pi \I}n} \big\rvert}{\big\lvert \E^{\frac1n V}  \big\rvert}.
\end{align*}

The denominator $ \big\lvert \E^{\frac1n V}  \big\rvert$ is bounded away from 0 by a positive constant depending only on $V$. 
Because $-\frac\pi2 -\delta \leq \frac{2\pi j}n \leq -\frac\pi2+ \delta$ or $\frac\pi2 -\delta \leq \frac{2\pi j}n \leq \frac\pi2+ \delta$, the numerator $\big\lvert  1  + \E^{ \frac 1n U - \frac{4\pi j}n \I - \frac{4\pi \I}n} \big\rvert$ is close to $0$ for $\delta$ small and $n$ large. It follows that the above ratio is less than $1$, and therefore that $\left\lvert  \QDL^q (u, v \vbar j)   \right\rvert$ is a decreasing function of $j$ if $\delta$ is small enough and if $n$ is sufficiently large. 

As a consequence if $\frac\pi2 -\delta \leq \frac{2\pi j}n \leq \frac\pi2+ \delta$ or, equivalently, $\frac n4 - \frac{n\delta}{2\pi} \leq j \leq \frac n4 + \frac{n\delta}{2\pi}$, then
\begin{align*}
 \left\lvert  \QDL^q (u, v \vbar j)   \right\rvert
 &\leq \left\lvert  \QDL^q \left(u, v \vbar j \left\lfloor\textstyle  \frac n4 - \frac{n\delta}{2\pi} \right\rfloor \right )   \right\rvert 
\end{align*}
where the floor function $\lfloor a \rfloor $ denotes the largest integer less than or equal to $a$. Since $\frac\pi2 -\delta -\frac{2\pi}n \leq \frac{2\pi}n \left\lfloor  \frac n4 - \frac{n\delta}{2\pi} \right\rfloor   \leq \frac\pi2 -\delta $, we can apply Lemma~\ref{lem:EstimateQdlFarFromPi/2} to estimate $ \left\lvert  \QDL^q \left(u, v \vbar  \left\lfloor\textstyle  \frac n4 - \frac{n\delta}{2\pi} \right\rfloor \right )   \right\rvert $ and conclude that
\begin{align*}
 \left\lvert  \QDL^q (u, v \vbar j)  \right\rvert
 & = O\left(  \E^{\frac{n}{2\pi} \Lambda \left( \frac\pi2-  \frac{2\pi}n \left\lfloor  \frac n4 - \frac{n\delta}{2\pi} \right\rfloor \right ) }  \right)
 \\
 & = O\left(  \E^{\frac{n}{2\pi} \Lambda \left( \delta + \frac{2\pi}n \right ) }  \right) 
  = O\left(  \E^{\frac{n}{2\pi} \Lambda \left(2 \delta \right ) }  \right)
\end{align*}
for $\delta$ sufficiently small and $n$ sufficiently large, using for the second and third equalities the property that the Lobachevsky function $\Lambda$ is increasing on a neighborhood of 0. This proves the required estimate when $\frac\pi2 -\delta \leq \frac{2\pi j}n \leq \frac\pi2+ \delta$. 

Similarly, when $ -\frac \pi2 -\delta \leq \frac{2\pi j}n \leq -\frac\pi2+ \delta$, 
\begin{align*}
 \left\lvert  \QDL^q (u, v \vbar j)   \right\rvert
 &\leq \left\lvert  \QDL^q \left(u, v \vbar  \left\lfloor\textstyle - \frac n4 - \frac{n\delta}{2\pi} \right\rfloor \right )   \right\rvert 
=  \left\lvert  \QDL^q \left(u, v \vbar  \left\lfloor\textstyle  \frac {3n}4 - \frac{n\delta}{2\pi} \right\rfloor \right )   \right\rvert ,
\end{align*}
using the periodicity of $\QDL^q (u, v \vbar j )$ modulo $n$. Then, $\frac{3\pi}2 -\delta -\frac{2\pi}n \leq \frac{2\pi}n \left\lfloor  \frac {3n}4 - \frac{n\delta}{2\pi} \right\rfloor   \leq \frac{3\pi}2 -\delta $ and applying Lemma~\ref{lem:EstimateQdlFarFromPi/2}  again gives
$$
\left\lvert  \QDL^q (u, v \vbar j)   \right\rvert
 = O\left(  \E^{\frac{n}{2\pi} \Lambda \left(-\pi + \delta + \frac{2\pi}n \right ) }  \right) 
  = O\left(  \E^{\frac{n}{2\pi} \Lambda \left(2 \delta \right ) }  \right).
$$
This concludes the proof in this case as well.
\end{proof}

\subsection{A general lemma}
\label{sect:AsymptoticsGeneralRealSum}

After the estimates of \S \ref{sect:QuantumDilogEstimates}, our asymptotic analysis of the sum $\Sigma_n$ will be based on  the following elementary (and relatively classical) lemma.

\begin{lem}
\label{lem:SumRealExponentials}
 Let the continuous function $f \colon [a,b] \to \R$ attain its maximum at a unique point $x_0 \in \left]a,b\right[$, and let $g_n \colon [a,b] \to \C$ be a sequence of continuous functions that uniformly converges to $g_\infty \colon [a,b] \to \C$ with $g_\infty(x_0)\neq 0$. Suppose that $f$ is twice differentiable on a neighborhood of $x_0$ and that $f''(x_0)\neq 0$. Then, as  $n$ tends to $\infty$, 
\begin{align}
\label{eqn:SumRealExponentialsPlain}
& \sum_{j \in \Z, \ a \leq \frac {2\pi j}n \leq b} g_n\left( \textstyle \frac {2\pi j}n \right) \E^{n f \left(\frac {2\pi j}n \right)} 
 \asymp {\textstyle \frac{g_\infty(x_0) }{\sqrt{- 2\pi f''(x_0)}}} \sqrt n \, \E^{n f(x_0)}
 \\
 \label{eqn:SumRealExponentialsAlternating}
 \text{and }
&  \sum_{j \in \Z,\  a \leq \frac {2\pi j}n \leq b} (-1)^j g_n\left( \textstyle \frac {2\pi j}n \right) \E^{n f \left(\frac {2\pi j}n \right)} 
 \prec \sqrt n\,  \E^{n f(x_0)}. 
\end{align}
\end{lem}

Here the symbols $A\asymp B$ and $A' \prec B'$ respectively mean that, as $n$ tends to $\infty$, the ratio $A/B$ converges to 1 while $A'/B'$ converges to 0. 

In the above estimates, it is crucial that the function $f$ be real valued. 

\begin{proof} By a change of variable we can assume that $x_0=0$, which will alleviate the notation. 
We can then  write
$$
f(x) = f(0) + x^2 h(x)
$$
for a continuous function $h \colon [a,b] \to \R$ with $h(0)=\frac12 f''(0)<0$. 

Also, since $f'(0)=0$ and $f''(0)<0$, there exists $a'<0$ and $b'>0$ such that $f$ is increasing on $[a', 0]$ and decreasing on $[0, b']$. We first consider the terms in $[a', b']$. Then,
\begin{align*}
 {\textstyle \frac {2\pi}{\sqrt n}} \E^{-n f \left(0\right)} 
  \sum_{a' \leq \frac {2\pi j}n \leq b'} g_n\left( \textstyle \frac {2\pi j}n \right) \E^{n f \left(\frac {2\pi j}n \right)} 
 & =
  {\textstyle \frac {2\pi }{\sqrt n}}  \sum_{a' \leq \frac {2\pi j}n \leq b'} g_n \left( \textstyle \frac {2\pi j}n \right) \E^{ \frac{4\pi^2 j^2}n h\left( \frac {2\pi j}n \right)  } 
  \\
&=
\int_{-\infty}^{+ \infty} F_n(x)\, dx
\end{align*}
where $F_n \colon \R \to \C$ is the step function defined by
$$
F_n(x) = 
\begin{cases}
  g_n \left( \textstyle \frac {2\pi j}n \right) \E^{ \frac{4\pi^2 j^2}n  h \left(  \frac {2\pi j}n \right)    } 
  &\text{if } \frac{2\pi j}{\sqrt n} \leq x < \frac{2\pi(j+1)}{\sqrt n} \text{ for some } j\in \Z 
  \\
  & \qquad\qquad\qquad\qquad\qquad \text{ with } {a' \leq \frac {2\pi j}{ n} \leq b' },\\
  0 & \text{otherwise}.
\end{cases}
$$

For a fixed $x$, if  $ \frac{2\pi j}{\sqrt n} \leq x < \frac{2\pi(j+1)}{\sqrt n}$, then $ \frac x{\sqrt n} - \frac{2\pi}{ n}<\frac{2\pi j}{ n} \leq \frac x{\sqrt n} $ and therefore $\frac{2\pi j}n$  tends to $0$ as $n$ tends to $\infty$. It follows that  the function $F_n$  simply converges to the function
$$
F_\infty(x) = g_\infty\left(0\right) \E^{x^2 h(0)}
$$
as $n$ tends to $\infty$. 

Because  the functions $g_n \colon [a,b] \to \C$ uniformly converge to the continuous function $g_\infty$, their moduli $\left\lvert g_n \right\rvert$ are uniformly bounded by a constant $M$. Also, $h(x)<0$ for every $x\neq 0$ since  $f(x)<f(0)$,  while $h(0) = \frac12 f''(0) <0$ is also negative by hypothesis. It follows that there exists $\epsilon>0$ such that $h(x) \leq -\epsilon$ for every $x\in [a,b]$. 

If $x\geq \frac{4\pi}{\sqrt n}$, it follows from the definition of $F_n(x)$ and from the fact that the function $y \mapsto y^2 h(y) = f(y)-f(0)$ is decreasing on $[0,b']$ that
$$
\left\lvert F_n(x) \right\rvert 
\leq M \E^{n\left(\frac x{\sqrt n} - \frac{2\pi}n \right)^2  h\left( \frac x{\sqrt n} - \frac{2\pi}n \right)}
\leq M \E^{ - \left( x - \frac{2\pi}{\sqrt n} \right)^2 \!\epsilon}
\leq M \E^{ -\frac\epsilon4 x^2 },
$$
using the fact that $  \frac{2\pi j}{ n} > \frac x{\sqrt n} - \frac{2\pi}n$ when $ \frac{2\pi j}{\sqrt n} \leq x < \frac{2\pi(j+1)}{\sqrt n}$, and  the na\"\i ve property that $x-\frac{2\pi}{\sqrt n} \geq \frac12 x$ when $x\geq \frac{4\pi}{\sqrt n}$ for the last inequality.

If $x\leq 0$, the fact that the function $y \mapsto y^2 h(y)$ is increasing on $[a',0]$ similarly gives
$$
\left\lvert F_n(x) \right\rvert 
\leq M \E^{n \left( \frac x{\sqrt n} \right)^2 h\left( \frac x{ \sqrt n}  \right)}
\leq M \E^{ -\epsilon x^2}.
$$

Finally, if $0 \leq x < \frac{4\pi}{\sqrt n}$, by definition $F_n(x)$ is equal to $g_n(0)$ or  $g_n\left( \textstyle\frac{2\pi}n \right) \E^{\frac{4\pi^2}n  h\left( \frac {2\pi}n \right)} $. Since $h\left( \frac {2\pi}n \right)<0$, it follows that its absolute value $\left\lvert F_n(x) \right\rvert $ is bounded by $M$. 

Therefore, in all cases, $\left\lvert F_n(x) \right\rvert $ is bounded by the integrable function
$$
F(x) =
\begin{cases}
M \E^{ -\epsilon x^2} &\text{if } x<0 \\
M &\text{if } 0\leq x<\frac{4\pi}{\sqrt n} \\
M \E^{ -\frac\epsilon4 x^2 } &\text{if } x\geq \frac{4\pi}{\sqrt n}.
\end{cases}
$$

By Lebesgue's Dominated Convergence Theorem, it follows that 
\begin{align*}
\lim_{n\to \infty}
 {\textstyle \frac {2\pi}{\sqrt n}} \E^{-n f \left(0\right)} 
  \sum_{a' \leq \frac {2\pi j}n \leq b'} g_n\left( \textstyle \frac {2\pi j}n \right) \E^{n f \left(\frac {2\pi j}n \right)} 
  &=
  \lim_{n\to \infty}
\int_{-\infty}^{+ \infty} F_n(x)\, dx
\\
&= \int_{-\infty}^{+ \infty} F_\infty(x)\, dx = g_\infty(0) \int_{-\infty}^{+ \infty}\E^{x^2 h(0)} dx
\\
&=  \textstyle g_\infty(0) \frac{\sqrt \pi}{\sqrt{-h(0)}} = \frac{\sqrt{2\pi} g_\infty(0)}{\sqrt{-f''(0)}}. 
\end{align*}

 This proves that 
 $$
   \sum_{ a' \leq \frac {2\pi j}n \leq b'} g_n\left( \textstyle \frac {2\pi j}n \right) \E^{n f \left(\frac {2\pi j}n \right)} 
 \asymp {\textstyle \frac{g_\infty(0) }{\sqrt{- 2\pi f''(0)}}} \sqrt n \, \E^{n f(0)} .
 $$
 
 We still have to take care of the terms that are in the intervals $\left [a,a'\right[$ and $\left]b', b\right]$. Since $f$ achieves its maximum only at $0$, there exists $\epsilon>0$ such that $f(x) \leq f(0)- \epsilon$ for every $x$ in $\left [a,a'\right[$ or $\left]b', b\right]$. Also, remember that the moduli $|g_n(x)|$ are uniformly bounded by a constant $M$. Then
$$
  \left|  \sum_{ a \leq \frac {2\pi j}n < a'} g_n\left( \textstyle \frac {2\pi j}n \right) \E^{n f \left(\frac {2\pi j}n \right)} \right|
\leq \left(\textstyle\frac{a'-a}{2\pi} n +1 \right) M \E^{n(f(0) - \epsilon)} \prec \E^{n f(0)} .
 $$ 
 A similar property holds for the interval $\left]b', b\right]$. 
 
 Combining these estimates for the contributions of the intervals $[a', b']$, $\left [a,a'\right[$ and $\left]b', b\right]$, we obtain that
  $$
   \sum_{ a \leq \frac {2\pi j}n \leq b} g_n\left( \textstyle \frac {2\pi j}n \right) \E^{n f \left(\frac {2\pi j}n \right)} 
 \asymp {\textstyle \frac{g_\infty(0) }{\sqrt{- 2\pi f''(0)}}} \sqrt n \, \E^{n f(0)} .
 $$
 which is our required  estimate (\ref{eqn:SumRealExponentialsPlain}) (when $x_0=0$).

To prove (\ref{eqn:SumRealExponentialsAlternating}), we again restrict attention to the contribution of the interval $[a', b']$. We then group consecutive terms in the alternating sum, and consider
$$
S_0 = \sum_{a' \leq \frac {4\pi k}n \leq b'-\frac {2\pi}n} \left( g_n \left( \textstyle \frac {4\pi k}n \right) \E^{n f \left(\frac {4\pi k}n \right)} - g_n\left( \textstyle \frac {4\pi k}n + \frac{2\pi}n \right) \E^{n f \left(  \frac {4\pi k}n + \frac{2\pi}n  \right)}  \right).
$$

We begin with an algebraic manipulation, and split this sum as $S_0 = S_1 + S_2$ with
\begin{align*}
S_1 &=  \sum_{a' \leq \frac {4\pi k}n \leq b'-\frac {2\pi}n} g_n \left(\textstyle \frac{4\pi k}n \right) \check f_n \left(\textstyle \frac{4\pi k}n \right) \E^{n f\left( \frac{4\pi k}n \right)}
\\
S_2 &=  \sum_{a' \leq \frac {4\pi k}n \leq b'-\frac {2\pi}n} \check g_n \left(\textstyle  \frac {4\pi k}n + \frac{2\pi}n  \right)  \E^{n f\left(  \frac {4\pi k}n + \frac{2\pi}n  \right)}
\end{align*}
for the functions
\begin{align*}
 \check f_n (x) &= 1- \E^{n \left( f \left(x + \frac{2\pi}n \right) - f\left(x  \right) \right) }
 \\
 \check g_n (x) &= g_n\left( \textstyle x-\frac{2\pi}n \right) - g_n(x).
\end{align*}

As $n$ tends to infinity, the function $\check  f_n (x)$ uniformly converges to $1-\E^{2\pi f'(x)}$. The argument that we used  to prove (\ref{eqn:SumRealExponentialsPlain}) then shows that
$$
\lim_{n\to \infty} \textstyle \frac{4\pi}{\sqrt n} \E^{-nf(0)} S_1 =\frac{\sqrt{2\pi}}{\sqrt{-f''(0)}} \, g_\infty(0) \left( 1-\E^{2\pi f'(0)} \right) =0
$$
since $f'(0)=0$. 

Similarly, $\check  g_n(x)$ uniformly converges to 0 as $n$ tends to $\infty$. Again, this implies that
$$
\lim_{n\to \infty} \textstyle \frac{4\pi}{\sqrt n} \E^{-nf(0)} S_2  =0.
$$

The combination of these two properties shows that $\frac1{\sqrt n} \E^{-nf(0)} S_0$ tends to 0 as $n$ tends to $\infty$. In other words, $S_0 \prec \sqrt n\, \E^{nf(0)}$. 

For the original sum of (\ref{eqn:SumRealExponentialsAlternating}), the difference
$$
  \sum_{a \leq \frac {2\pi j}n \leq b} (-1)^j g\left( \textstyle \frac {2\pi j}n \right) \E^{n f \left(\frac {2\pi j}n \right)} -S_0
$$
consists of approximately $\frac{a'-a+b-b'}{2\pi}n$ terms whose modulus is bounded by $M \E^{n(f(0)-\epsilon)}$. As in the proof of  (\ref{eqn:SumRealExponentialsPlain}), the contribution of these remaining terms is $\prec \E^{nf(0)}$. With our estimate for $S_0$, this proves (\ref{eqn:SumRealExponentialsAlternating}).

This completes the proof of (\ref{eqn:SumRealExponentialsPlain}) and (\ref{eqn:SumRealExponentialsAlternating}) when $x_0=0$. As indicated at the beginning, the general case follows by a change of variable. 
\end{proof}

\subsection{Asymptotics of the sum $\Sigma_n$}
\label{subsect:AsymptoticsSn}

We now have all the tools we need to find the asymptotics, as $n$ tends to infinity,  of the sum 
$$\Sigma_n = \sum_{i=1}^n \QDL^q (u,v \vbar 2i) q^{2i^2 - \widehat k i}$$
of (\ref{eqn:TheSumSn}). 

\begin{prop}
\label{prop:AsymptoticsSn}
Let $U$, $V \in \C$ with $\E^V=1+\E^U$ be given, as well as an integer $\widehat k \in \Z$. For every odd $n$, set $q = \E^{\frac{2\pi\I}n}$, $u=\E^{\frac1n U}$ and $v=\E^{\frac1n V}$. Then, 
$$
\sum_{i=1}^n \QDL^q (u,v \vbar 2i) q^{2i^2 - \widehat k i}
 \asymp c_n(U,V, \widehat k)
\sqrt n \, \E^{\frac n{2\pi} \Lambda\left( \frac\pi6 \right)}
$$
as $n\to \infty$, where
$$
c_n\big(U,V, \widehat k \big) =
\begin{cases}
 3^{-\frac14}
2^{\frac{U-2\pi\I}{4\pi\I}}\ 
\E^{\frac{-\widehat k \pi\I }6  } \ 
 \E^{ - \frac {V}{6}} 
 \left( \textstyle \frac{1-\I \sqrt3}2 \right)^{\frac{2\pi\I - U}{4\pi\I}}  
 &\text{ if } \widehat k \text{ is odd,}
 \\
 3^{-\frac14}
2^{\frac{U-2\pi\I}{4\pi\I}}\ 
 \E^{-\frac{n \pi\I}4}\  
 \E^{\frac{-2\widehat k \pi\I}3 }\ 
  \E^{ - \frac {2 V}{3}}\ 
\big( 1 - \I^n \E^{ \frac 12 U }  \big)
 \left( \textstyle \frac{1-\I \sqrt3}2 \right)^{\frac{2\pi\I - U}{4\pi\I}}
  &\text{ if } \widehat k \text{ is even.}
\end{cases}
$$
\end{prop}

Note that $c_n\big(U,V, \widehat k \big)$ is independent of $n$ when $\widehat k$ is odd, and depends only on the congruence class of $n$ modulo 8 when $\widehat k$ is even. In addition,  the modulus $\left| c_n \big( U,V, \widehat k \big )\right|$ depends only on $n$ modulo 4 in this second case. 

\begin{proof}  Let us focus attention on the case where $\widehat k$ is even. The argument will be essentially identical when $\widehat k$ is odd.

We first rewrite the sum $\Sigma_n$ in a slightly simpler form, by an algebraic manipulation. For this, let $\omega$ be the (unique) square root of $q$ such that $\omega^n=1$, namely $\omega = - \E^{\frac{\pi\I}n}$. Then the function $j \mapsto \QDL^q (u,v \vbar j) \omega^{j^2 - \widehat k j}$ is $n$--periodic. As $i$ goes from 1 to $n$, $j=2i$ also ranges over all numbers from $1$ to $n$ modulo $n$, since $n$ is odd. This enable us to rewrite the sum $\Sigma_n$ as
\begin{equation}
\Sigma_n = \sum_{i=1}^n \QDL^q (u,v \vbar 2i) \omega^{4i^2 - 2 \widehat k i} = \sum_{j=1}^n \QDL^q (u,v \vbar j) \omega^{j^2 - \widehat k j}.
\end{equation}

Choose a small $\delta>0$ as in Lemma~\ref{lem:EstimateQdlNearPi/2}. Then, using the periodicity of its terms, we can write the sum 
$$
\Sigma_n =  \sum_{-\frac\pi2-\delta \leq \frac{2\pi j}n < \frac{3\pi}2 - \delta} \QDL^q (u,v \vbar j) \omega^{j^2 - \widehat k j}
$$
as $\Sigma_n = S_1 + S_2 + S_3 + S_4$ with
\begin{align*}
S_1 &= 
 \kern-10pt   \sum_{-\frac\pi2-\delta \leq \frac{2\pi j}n < -\frac{\pi}2 + \delta}  \kern-10pt 
 \QDL^q (u,v \vbar j) \omega^{j^2 - \widehat k j}
 &
 S_2 &= 
 \kern-10pt  \sum_{-\frac\pi2+\delta \leq \frac{2\pi j}n < \frac\pi2- \delta} \kern-10pt 
  \QDL^q (u,v \vbar j) \omega^{j^2 - \widehat k j}
 \\
 S_3 &= 
 \kern-10pt  \sum_{\frac\pi2-\delta \leq \frac{2\pi j}n < \frac{\pi}2 + \delta} \kern-10pt 
  \QDL^q (u,v \vbar j) \omega^{j^2 - \widehat k j}
 &
 S_4 &= 
 \kern-10pt  \sum_{\frac\pi2+\delta \leq \frac{2\pi j}n <\frac{3\pi}2 - \delta } \kern-10pt 
 \QDL^q (u,v \vbar j) \omega^{j^2 - \widehat k j} .
\end{align*}

By Lemma~\ref{lem:EstimateQdlNearPi/2}, each term of the sum $S_1$ is an $O\left(  \E^{\frac{n}{2\pi} \Lambda \left(2\delta \right ) }  \right)$. Since $S_1$ has $\left\lfloor \frac\delta\pi n \right\rfloor$ terms and $\Lambda(2\delta) < \Lambda(3\delta)$, if follows that
\begin{equation}
\label{eqn:EstimateSumS1}
 S_1 = O\left(  \E^{\frac{n}{2\pi} \Lambda \left(3\delta \right ) }  \right).
\end{equation}

Similarly
\begin{equation}
\label{eqn:EstimateSumS3}
 S_3 = O\left(  \E^{\frac{n}{2\pi} \Lambda \left(3\delta  \right ) }  \right).
\end{equation}

We now consider the sum $S_2$. By the first statement of Lemma~\ref{lem:EstimateQdlFarFromPi/2}, 
\begin{align*}
 S_2 &= 
 \sum_{-\frac\pi2+\delta \leq \frac{2\pi j}n < \frac\pi2- \delta}
  \QDL^q (u,v \vbar j) \omega^{j^2 - \widehat k j}
 \\
 &=
  \sum_{-\frac\pi2+\delta \leq \frac{2\pi j}n < \frac{\pi}2 - \delta} 
   \E^{n \widetilde  f \left( \frac {2\pi j} n \right)} \ 
  \widetilde g_n \left(\textstyle  \frac {2\pi j} n \right) 
  (-1)^{j^2 - \widehat k j}\ 
  \E^{\frac{\pi \I}n (j^2 - \widehat k j)}
  \\
  &=  
   \sum_{-\frac\pi2+\delta \leq \frac{2\pi j}n < \frac{\pi}2 - \delta} 
     (-1)^j \E^{n   f \left( \frac {2\pi j} n \right)} \ 
   g_n \left(\textstyle  \frac {2\pi j} n \right) 
\end{align*}
for the functions $\widetilde f$ and $\widetilde g_n$ of Lemma~\ref{lem:EstimateQdlFarFromPi/2}, and if we set $f(t) = \widetilde f(t) + \frac 1{4\pi}t^2 \I$ and $g_n(t) = \widetilde g_n(t) \E^{-\frac{\widehat k}2  t\I}$. We are here using the fact that $\widehat k$ is even, so that $  (-1)^{j^2 - \widehat k j} = (-1)^j$. 

In particular, $f(t) = \widetilde f(t) + \frac 1{4\pi}t^2 \I = \frac1{2\pi} \Lambda\left( \frac\pi2 - u \right) $.

\begin{rem}
 \label{rem:NoPhase}
 We pause here to emphasize a critical outcome of this computation, which is that $f(t) = \frac1{2\pi} \Lambda\left( \frac\pi2 - u \right) $ is real. This property will enable us to apply  Lemma~\ref{lem:SumRealExponentials}, but is very specific to the special case $\phi=LR$ considered here. In the case of a general diffeomorphism $\phi$, the existence of a phase leads to cancellations requiring more sophisticated methods, such as the saddle point method used in \cite{BWY3}. 
\end{rem}

We now return to the proof, and apply Lemma~\ref{lem:SumRealExponentials}. 
Over the interval $\left [-\frac{\pi}2 + \delta,  \frac{\pi}2 - \delta \right]$, the function $f(t)= \frac1{2\pi} \Lambda\left( \frac\pi2 - t \right)$ admits a unique maximum at $t= \frac\pi3$, where $f\left( \frac \pi 3 \right) = \frac1{2\pi} \Lambda\left( \frac \pi 6 \right)$ and $f''\left( \frac \pi 3 \right) = -\frac{\sqrt3}{2\pi}<0$. The function 
$$
g_n(t) = 2^{\frac{U-2\pi\I}{4\pi\I}}\ 
 \E^{-\frac{\widehat k}2  t\I}\ 
 \left(1+ \E^{- 2t \I} \right)^{\frac{2\pi\I - U}{4\pi\I}} 
 \E^{ - \frac {t V}{2\pi}} 
\E^{ O \left( \frac1{n} \right)}
$$
 uniformly converges to 
$$
g_\infty(t) =   2^{\frac{U-2\pi\I}{4\pi\I}}\ 
 \E^{-\frac{\widehat k}2  t\I}\ 
 \left(1+ \E^{- 2t \I} \right)^{\frac{2\pi\I - U}{4\pi\I}} 
 \E^{ - \frac {t V}{2\pi}}  .
$$
We can therefore apply Lemma~\ref{lem:SumRealExponentials}, which shows that 
\begin{equation}
\label{eqn:EstimateSumS2even}
S_2 \prec 
\sqrt n \, \E^{\frac n{2\pi} \Lambda\left( \frac\pi6 \right)}.
\end{equation}

Finally, we can tackle $S_4$. By the second statement of Lemma~\ref{lem:EstimateQdlFarFromPi/2}, 
\begin{align*}
 S_4 &= 
 \sum_{\frac\pi2+\delta \leq \frac{2\pi j}n < \frac{3\pi}2 - \delta} 
  \QDL^q (u,v \vbar j) \omega^{j^2 - \widehat k j}
 \\
 &=
   \sum_{\frac\pi2+\delta \leq \frac{2\pi j}n < \frac{3\pi}2 - \delta} (-1)^j \E^{n \widehat  f \left( \frac {2\pi j} n \right)} \ 
  \widehat g_n \left(\textstyle  \frac {2\pi j} n \right) 
  (-1)^{j^2 - \widehat k j}\ 
  \E^{\frac{\pi \I}n (j^2 - \widehat k j)}
  \\
  &=  \sum_{\frac\pi2+\delta \leq \frac{2\pi j}n < \frac{3\pi}2 - \delta}  \E^{n   f \left( \frac {2\pi j} n \right)} \ 
   g_n \left(\textstyle  \frac {2\pi j} n \right) 
\end{align*}
for the functions $\widehat f$ and $\widehat g_n$ of Lemma~\ref{lem:EstimateQdlFarFromPi/2}, and where this time $f(t) = \widehat f(t) + \frac 1{4\pi}t^2 \I$ and $g_n(t) = \widehat g_n(t) \E^{-\frac{\widehat k}2  t\I}$.

Again, the function $ f(t)= \frac1{2\pi} \Lambda\left( \frac{3\pi}2 - t \right)$ admits a unique maximum over the interval $\left [\frac\pi2+\delta ,  \frac{3\pi}2 - \delta \right]$, at $t= \frac{4\pi}3$ where $ f\left( \frac {4\pi} 3 \right) = \frac1{2\pi} \Lambda\left( \frac \pi 6 \right)$ and $ f''\left( \frac {4\pi} 3 \right) = -\frac{\sqrt3}{2\pi}$. We cannot quite apply  Lemma~\ref{lem:SumRealExponentials} as stated, as  the function 
$$
g_n(t)= 
2^{\frac{U-2\pi\I}{4\pi\I}}\ 
 \E^{-\frac{n\pi\I}4}\  
 \E^{-\frac{\widehat k}2  t\I}\ 
\big( 1 - \I^n \E^{ \frac 12 U }  \big)
 \left(1+ \E^{- 2t \I} \right)^{\frac{2\pi\I - U}{4\pi\I}} 
 \E^{ - \frac {t V}{2\pi}}
\E^{ O \left( \frac1{n} \right)} 
$$
does not have one, but four limits as $n$ tends to infinity. More precisely, as the odd integer $n$ tends to infinity while $n = n_0 \ \mathrm{mod}\  8$ for a fixed $n_0=1$, $3$, $ 5$ of $7$, the function $ g_n$ uniformly converges to 
$$
 g_{\infty}^{[n_0]}(t)
= 2^{\frac{U-2\pi\I}{4\pi\I}}\ 
 \E^{-\frac{n_0 \pi\I}4}\  
 \E^{-\frac{\widehat k}2  t\I}\ 
\big( 1 - \I^{n_0} \E^{ \frac 12 U }  \big)
 \left(1+ \E^{- 2t \I} \right)^{\frac{2\pi\I - U}{4\pi\I}} 
 \E^{ - \frac {t V}{2\pi}}.
$$

In particular,
$$
 g_{\infty}^{[n_0]}\left(\textstyle \frac{4\pi}3 \right) 
= 2^{\frac{U-2\pi\I}{4\pi\I}}\ 
 \E^{-\frac{n_0 \pi\I}4}\  
 \E^{-\frac{2\widehat k \pi \I}3 }\ 
  \E^{ - \frac {2 V}{3}}\ 
\big( 1 - \I^{n_0} \E^{ \frac 12 U }  \big)
 \left( \textstyle \frac{1-\I \sqrt3}2 \right)^{\frac{2\pi\I - U}{4\pi\I}} .
$$

Then, applying Lemma~\ref{lem:SumRealExponentials} to appropriate subsequences shows that
\begin{equation}
\label{eqn:EstimateSumS4even}
S_4 \asymp 
 {\textstyle \frac{g_\infty^{[n_0]} \left( \frac{4\pi}3 \right)  }{\sqrt{- 2\pi f'' \left( \frac{4\pi}3 \right) }}} \sqrt n \, \E^{n f \left( \frac{4\pi}3 \right) }
=c_n \big(U,V, \widehat k \big)  \sqrt n \ \E^{\frac n{2\pi} \Lambda\left( \frac\pi6 \right)} 
\end{equation}
with
$$
c_n \big(U,V, \widehat k \big) = c_{n_0} \big(U,V, \widehat k \big)=
3^{-\frac14}
2^{\frac{U-2\pi\I}{4\pi\I}}\ 
 \E^{-\frac{n \pi\I}4}\  
 \E^{-\frac{2\widehat k \pi \I}3 }\ 
  \E^{ - \frac {2 V}{3}}\ 
\big( 1 - \I^{n} \E^{ \frac 12 U }  \big)
 \left( \textstyle \frac{1-\I \sqrt3}2 \right)^{\frac{2\pi\I - U}{4\pi\I}}  .
$$

Now, in the sum
$$
\Sigma_n = S_1 + S_2 + S_3 + S_4,
$$
the estimates (\ref{eqn:EstimateSumS1}), (\ref{eqn:EstimateSumS3}), (\ref{eqn:EstimateSumS2even}) and (\ref{eqn:EstimateSumS4even}) then show that $S_1$, $S_2$ and $S_3$ are negligible compared  to $S_4$. Therefore, when $\widehat k$ is even, 
$$
S \asymp S_4 \asymp c_{n} \big(U,V, \widehat k \big) \sqrt n\  \E^{\frac n{2\pi} \Lambda\left( \frac\pi6 \right)} 
$$
as $n$ tends to infinity.

The case when $\widehat k$ is odd is very similar, except that the roles of $S_2$ and $S_4$ are exchanged. The leading term is now $S_2$, which is somewhat simpler. In this case, the function $ f(t)= \frac1{2\pi} \Lambda\left( \frac{\pi}2 - t \right)$ admits its maximum over the interval $\left [-\frac\pi2+\delta ,  \frac{\pi}2 - \delta \right]$ at $t= \frac{\pi}3$, and the same arguments as above give
$$
S \asymp S_2 \asymp c_{n} \big(U,V, \widehat k \big) \sqrt n \ \E^{\frac n{2\pi} \Lambda\left( \frac\pi6 \right)} 
$$
with 
\begin{equation*}
c_{n} \big (U,V, \widehat k \big) =  
3^{-\frac14}
2^{\frac{U-2\pi\I}{4\pi\I}}\ 
\E^{-\frac{\widehat k \pi\I }6 } \ 
 \E^{ - \frac {V}{6}} 
 \left( \textstyle \frac{1-\I \sqrt3}2 \right)^{\frac{2\pi\I - U}{4\pi\I}}   
\end{equation*}
independent of $n$. 
\end{proof}

\section{The limit of $\left| D^q(u) \right|^{\frac1n}$}
\label{sect:AsymptoticsDq}

We now consider the factor
$$
D^q(u) = \prod_{i=1}^n \QDL( u,v \vbar i) = (1+u^n)^{-\frac{n+1}2} \prod_{i=1}^n (1 + u q^{-2i})^{n-i+1}
$$
defined when $u^n =v^n-1\neq -1$. We are interested in the limit of $\left| D^q(u )  \right|^{\frac1n}$ as $n$ odd tends to $\infty$, with $q = \E^{\frac{2\pi\I}n}$ and $u = \E^{\frac1n U}$ for some $U$ such that $\E^U \neq -1$. 

It will actually be more convenient to set $A=2\pi \I -U$, so that $u=q \E^{\frac1n A}$. There are two reasons for this. The first one is that this is the form under which the terms $\left| D^q(u_k )  \right|^{\frac1n} = \left| D^q \big( q \E^{-\frac1n A_k} \big)  \right|^{\frac1n}$ appear in the formula (\ref{eqn:FormulaTrace}) for $\Tr \Lambda_{\phi, r}^q$. The second reason is that the limit is more easily expressed in terms of $A$. In particular, we will see that $\left| D^q \big( q \E^{-\frac1n A} \big)  \right|^{\frac1n}$ is always equal to $1$ when $A$ is real. 

We will see that the modulus $ \left| D^q(u )  \right| $ has two limits as $n$ odd tends to $\infty$, according to the congruence of $n$ modulo $4$. We will determine these limits in three steps. The first one, in \S \ref{subsect:AlgebraicManipDq}, is a simple algebraic manipulation which enables us to rewrite   $\left| D^q(u )  \right| ^{\frac1n}$ as the product of two or three terms, according to the congruence of $n$ modulo $4$. When there are three terms, the limit of the extra term is completely straightforward. The next two steps, in \S \ref{subsect:FirstSumDq} and \S \ref{subsect:SecondSumDq},  are each devoted to the limit of one of the remaining  terms. 

\subsection{An algebraic manipulation}
\label{subsect:AlgebraicManipDq}

Some of the quantum dilogarithms $ \QDL( u,v \vbar i)$ grow exponentially with $n$ while others decrease exponentially. We first rearrange factors so that $D^q(u) $ is expressed as a product of terms reasonably close to 1, while also splitting this product as a preparation for the next steps. 

\begin{lem}
\label{lem:AlgebraicManipDq}
For every odd $n$, let $q= \E^{\frac{2\pi\I}n}$ and $u= q \E^{-\frac 1n A}$ for a fixed number $A\in \C$ with $\E^A \neq -1$. 

 If $n = 1  \ \mathrm{mod}\  4$, then for $m = \frac{n-1}4$
\begin{align*}
 \left| D^q(u )  \right| 
&=   \prod_{j=1}^m \frac
{ \left|  \E^{\frac { -A + 4\pi\I j - \pi\I} n  } -1 \right| ^{2j-1} }
{ \left|  \E^{\frac { - \bar A + 4\pi\I j - \pi\I} n  } -1 \right| ^{2j-1} }
\\
&\qquad\qquad\qquad\qquad
\prod_{j=1}^m \frac
{ \left|   \E^{\frac { -A + 4\pi\I j - \pi\I} n  }  -1 \right| ^{m-j+1}  \left|   \E^{\frac { - \bar A + 4\pi\I j -3 \pi\I} n  } -1\right| ^{m-j+1}}
{ \left|  \E^{\frac { - \bar A + 4\pi\I j - \pi\I} n  }  -1 \right| ^{m-j+1}  \left|   \E^{\frac { -A + 4\pi\I j - 3 \pi\I} n  } -1\right| ^{m-j+1}}.
\end{align*}

 If $n = 3  \ \mathrm{mod}\  4$, then for $m = \frac{n-3}4$
 \begin{align*}
 \left| D^q(u )  \right| 
&=   
\frac{\left|    \E^{\frac { -A +\pi\I} n  }  -1 \right| ^{m+1} }{ \left|  \E^{\frac { -\bar A +\pi\I} n  }  -1 \right| ^{m+1} }\
 \prod_{j=1}^m \frac
{ \left|  \E^{\frac { -A + 4\pi\I j + \pi\I} n  } -1 \right| ^{2j} }
{ \left|  \E^{\frac { - \bar A + 4\pi\I j + \pi\I} n  } -1 \right| ^{2j} }
\\
&\qquad\qquad\qquad\qquad
\prod_{j=1}^m \frac
{ \left|   \E^{\frac { -A + 4\pi\I j +\pi\I} n  }  -1 \right| ^{m-j+1}  \left|   \E^{\frac { - \bar A + 4\pi\I j - \pi\I} n  } -1\right| ^{m-j+1}}
{ \left|  \E^{\frac { - \bar A + 4\pi\I j + \pi\I} n  }  -1 \right| ^{m-j+1}  \left|   \E^{\frac { -A + 4\pi\I j -  \pi\I} n  } -1\right| ^{m-j+1}}.
\end{align*}

In particular,  in both cases  $ \left| D^q(u )  \right| =1$  when $A$ is real. 
\end{lem}

\begin{proof}
We begin by grouping together the quantum dilogarithms $  \QDL( u,v \vbar i)$ and $ \QDL( u,v \vbar n-i)$. Using the property that
$$
v^n = 1+u^n = \prod_{k=1}^n ( 1+ u q^{-2k} ) = \prod_{k=1}^{n-i} ( 1+ u q^{-2k} )  \prod_{j=1}^i ( 1+ u q^{2j-2} ) 
$$
(setting $j=n-k+1$ for the last term), we obtain that 
\begin{align*}
\big|  \QDL( u,v \vbar i) \QDL( u,v \vbar n-i) \big|
 &= 
\left| v \right|^{-i}   \prod_{j=1}^i \left| 1+ u q^{-2j}\right|  \left| v\right|^{-n+i} \prod_{k=1}^{n-i} \left| 1+ u q^{-2k} \right|
\\
&=
 \prod_{j=1}^i \frac{\left| 1+ u q^{-2j}\right| } {\left| 1+ u q^{2j-2}\right| }
 =
  \prod_{j=1}^i \frac{\left| 1+  u q^{-2j}\right| } {\left| 1+  \bar u q^{-2j+2}\right| }
=
 \prod_{j=1}^i \frac{\left| 1+ \E^{-\frac { A} n} q^{-2j+1}\right| } {\left| 1+  \E^{-\frac {\bar A} n} q^{-2j+1}\right| }. 
\end{align*}

Since  $\QDL( u,v \vbar n)=1$, can therefore rewrite 
\begin{align*}
\left| D^q(u )  \right| 
&= \prod_{i=1}^{\frac{n-1}2} 
\prod_{j=1}^i \frac{\left| 1+ \E^{-\frac { A} n} q^{- 2j+1}\right| } {\left| 1+  \E^{-\frac {\bar A} n} q^{- 2j+1}\right| }
= \prod_{i=1}^{\frac{n-1}2} \frac{\left| 1+ \E^{-\frac { A} n} q^{-2i+1} \right| ^{\frac{n-1}2 -i+1} } {\left| 1+  \E^{-\frac {\bar A} n} q^{-2i+1} \right|^{\frac{n-1}2 -i+1} }.
\end{align*}

We now group these terms into pairs of indices that are symmetric with respect to the midpoint of $\{1, 2, \dots, \frac{n-1}2\}$. 

Namely, when $\frac{n-1}2$ is even and equal to $2m$, we pair  $i=m-j+1$ with $i'=m+j$ for all $j \in \{1,2,\dots, m\}$. Note that 
\begin{align*}
 \left| 1+ \E^{-\frac { A} n} q^{-2i+1} \right| ^{\frac{n-1}2 -i+1} 
 &=
 \left| 1- \E^{\frac { -A+ 4\pi\I j - \pi\I} n } \right| ^{m+j} 
 =
 \left|  \E^{\frac { -A+ 4\pi\I j - \pi\I} n } -1 \right| ^{m+j}
 \\
 \text{ and } \left| 1+ \E^{-\frac { A} n} q^{-2i'+1} \right| ^{\frac{n-1}2 -i'+1} 
 &= \left| 1- \E^{ \frac { -A - 4\pi\I j + 3\pi\I} n } \right| ^{m-j+1} 
 = \left| \E^{ \frac { - \bar A + 4\pi\I j - 3\pi\I} n } -1\right| ^{m-j+1},
\end{align*}
and that the same equalities hold with $A$ replaced with $\bar A$. Therefore,
\begin{align*}
 \left| D^q(u )  \right| 
&=  \prod_{j=1}^m \frac
{ \left|   \E^{\frac { -A+ 4\pi\I j - \pi\I} n } -1 \right| ^{m+j}  \left|  \E^{ \frac { - \bar A + 4\pi\I j - 3\pi\I} n }  -1\right| ^{m-j+1}}
{ \left|    \E^{\frac { -\bar A+ 4\pi\I j - \pi\I} n } -1 \right| ^{m+j}  \left|  \E^{ \frac { -  A + 4\pi\I j - 3\pi\I} n } -1\right| ^{m-j+1}}
\\
&=   \prod_{j=1}^m \frac
{ \left| \E^{\frac { -A+ 4\pi\I j - \pi\I} n }-1 \right| ^{2j-1} }
{ \left|  \E^{\frac { -\bar A+ 4\pi\I j - \pi\I} n } -1 \right| ^{2j-1} }
\\
&\qquad\qquad\prod_{j=1}^m \frac
{ \left| \E^{\frac { -A+ 4\pi\I j - \pi\I} n } -1 \right| ^{m-j+1}  \left| \E^{\frac { -\bar A+ 4\pi\I j - 3\pi\I} n } -1\right| ^{m-j+1}}
{ \left| \E^{\frac { -\bar A+ 4\pi\I j - \pi\I} n }-1 \right| ^{m-j+1}  \left|  \E^{\frac { -A+ 4\pi\I j - 3\pi\I} n } -1\right| ^{m-j+1}}
\end{align*}
when $n = 4m+1$. 

When $\frac{n-1}2$ is odd and equal to $2m+1$, we pair  $i=m-j+1$ with $i'=m+j+1$ for all $j \in \{1,2,\dots, m\}$, which leaves alone the middle term corresponding to $i=m+1$. Then, by a computation similar to the one above, 
\begin{align*}
 \left| D^q(u )  \right| 
&=  
\frac{\left|  \E^{\frac {- A + \pi\I} n } -1 \right| ^{m+1} }{ \left|  \E^{\frac {- \bar  A + \pi\I} n  } -1 \right| ^{m+1} }\
\prod_{j=1}^m \frac
{ \left|  \E^{\frac {- A + 4\pi\I j + \pi\I} n } -1 \right| ^{m+j+1}  \left|  \E^{\frac {- \bar A + 4\pi\I j - \pi\I} n } -1\right| ^{m-j+1}}
{ \left|  \E^{\frac {- \bar A + 4\pi\I j + \pi\I} n }  -1 \right| ^{m+j+1}  \left|  \E^{\frac {- A + 4\pi\I j - \pi\I} n } -1\right| ^{m-j+1}}
\\
&=   
\frac{\left|  \E^{\frac {- A + \pi\I} n } -1 \right| ^{m+1} }{ \left|  \E^{\frac {- \bar  A + \pi\I} n  } -1 \right| ^{m+1} }\
\prod_{j=1}^m \frac
{ \left|  \E^{\frac {- A + 4\pi\I j + \pi\I} n }  -1 \right| ^{2j} }
{ \left|  \E^{\frac {- \bar A + 4\pi\I j + \pi\I} n } -1 \right| ^{2j} }
\\
&\qquad\qquad\qquad\qquad
\prod_{j=1}^m \frac
{ \left|  \E^{\frac {- A + 4\pi\I j + \pi\I} n } -1 \right| ^{m-j+1}  \left|  \E^{\frac {- \bar A + 4\pi\I j - \pi\I} n } -1\right| ^{m-j+1}}
{ \left|  \E^{\frac {- \bar A + 4\pi\I j + \pi\I} n }  -1 \right| ^{m-j+1}  \left|  \E^{\frac {- A + 4\pi\I j - \pi\I} n } -1\right| ^{m-j+1}}
\end{align*}
when $n = 4m+3$. 
\end{proof}

At this point, the split of the formulas of Lemma~\ref{lem:AlgebraicManipDq} as a product of two large products (and an isolated term when $n=3 \ \mathrm{mod}\ 4$) may look artificial. It is justified by the fact that, in \S \ref{subsect:FirstSumDq} and \S \ref{subsect:SecondSumDq}, we will be able to prove that each of the two products has a limit as $n$ tends to $\infty$, by rather different methods. 

Passing to logarithms,  Lemma~\ref{lem:AlgebraicManipDq} shows that, if $n=4m+1$, 
\begin{equation}
\label{eqn:FormulaLogDq4m+1}
\begin{aligned}
 \log  \left| D^q(u )  \right| ^{\frac1n}
&=   \sum_{j=1}^m \frac{2j-1}{4m+1}
\log \frac
{ \left|  \E^{\frac { -A + 4\pi \I j - \pi\I} {4m+1}  } -1 \right| }
{ \left|   \E^{\frac { -\bar A + 4\pi \I j - \pi\I} {4m+1}  } -1 \right| }
\\
&\qquad\qquad 
+ \sum_{j=1}^m \frac{m-j+1}{4m+1}
\log\frac{ \left|  \E^{\frac { -A + 4\pi \I j - \pi\I} {4m+1}  }  -1 \right|   \left|  \E^{\frac { -\bar A + 4\pi \I j -3 \pi\I} {4m+1}  } -1\right| }
{ \left|   \E^{\frac { -\bar A + 4\pi \I j - \pi\I} {4m+1}  }  -1 \right|   \left|  \E^{\frac { -A + 4\pi \I j - 3\pi\I} {4m+1}  } -1\right|}
\end{aligned}
\end{equation}
while, if $n=4m+3$,
\begin{equation}
\label{eqn:FormulaLogDq4m+3}
  \begin{aligned}
 \log  \left| D^q(u )  \right| ^{\frac1n}
&=   
\frac{m+1} {4m+3}
\log \frac{\left|  \E^{\frac { - A + \pi\I} {4m+3} } -1 \right| }{ \left|  \E^{\frac { - \bar A + \pi\I} {4m+3} }  -1 \right|  }\
+ \sum_{j=1}^m \frac{2j}{4m+3}
\log\frac { \left|  \E^{\frac { - A + 4\pi\I j + \pi\I }{4m+3}  } -1 \right| }
{ \left|   \E^{\frac { - \bar A + 4\pi\I j + \pi\I }{4m+3}  }  -1 \right|  }
\\
&\qquad\qquad
+\sum_{j=1}^m \frac{m-j+1}{4m+3}
\log\frac{ \left|   \E^{\frac { - A + 4\pi\I j + \pi\I }{4m+3}  }   -1 \right|   \left|  \E^{\frac { - \bar A + 4\pi\I j - \pi\I }{4m+3}  }  -1\right| }
{ \left|   \E^{\frac { -\bar  A + 4\pi\I j + \pi\I }{4m+3}  }  -1 \right|  \left|  \E^{\frac { - A + 4\pi\I j - \pi\I }{4m+3}  }  -1\right| }. 
\end{aligned}
\end{equation}

We begin with the easiest part, namely the limit of the first term of (\ref{eqn:FormulaLogDq4m+3}). 

\begin{lem}
\label{lem:IsolatedTermDq4m+3}
$$
\lim_{m\to \infty}
\frac{m+1} {4m+3}
\log \frac{\left|  \E^{\frac { - A + \pi\I} {4m+3} } -1 \right| }{ \left|  \E^{\frac { - \bar A + \pi\I} {4m+3} }  -1 \right|  }
 =
 \frac14\log \frac{\left|  A - \pi\I   \right| }{\left|  \bar A - \pi\I   \right| }.
$$
\end{lem}
\begin{proof} This is a simple consequence of the fact that
\begin{equation*}
 \E^{\frac { - A + \pi\I} {4m+3}}  -1 \ = \frac { - A + \pi\I} {4m+3}    +O\left( \frac1{m^2} \right). \qedhere
\end{equation*}
\end{proof}

\subsection{The limit of the first sums in (\ref{eqn:FormulaLogDq4m+1}--\ref{eqn:FormulaLogDq4m+3})}
\label{subsect:FirstSumDq}

We now determine the limits, as $m$ tends to $\infty$, of the first sums in (\ref{eqn:FormulaLogDq4m+1}) and (\ref{eqn:FormulaLogDq4m+3}). More precisely,

\begin{lem}
\label{lem:FirstSumDq}
\begin{align*}
\lim_{m\to \infty}
 \sum_{j=1}^m \frac{2j-1}{4m+1}
\log \frac
{ \left|  \E^{\frac { -A + 4\pi \I j - \pi\I} {4m+1}  } -1 \right| }
{ \left|   \E^{\frac { -\bar A + 4\pi \I j - \pi\I} {4m+1}  } -1 \right| }
&=  -  \frac{\log 2}{4\pi}  \Im A,
\\
\lim_{m\to \infty}
\sum_{j=1}^m \frac{2j}{4m+3}
\log\frac { \left|  \E^{\frac { - A + 4\pi\I j + \pi\I }{4m+3}  } -1 \right| }
{ \left|   \E^{\frac { - \bar A + 4\pi\I j + \pi\I }{4m+3}  }  -1 \right|  }
&=  -  \frac{\log 2}{4\pi}  \Im A. 
\end{align*}
\end{lem}

\begin{proof} We will focus attention on the first case, coming from (\ref{eqn:FormulaLogDq4m+1}). The other case will essentially be identical. 

We can write the sum as
$$
\frac1{4m+1} \sum_{j=1}^m f_m\left( \frac j{4m+1} \right)
$$
with
$$
f_m(t) = \left( 4m+1 \right) \left( 2t- \frac1{4m+1} \right)  \log \frac{\left| \E^{\frac {- A - \pi \I} {4m+1}+ 4 \pi \I t} -1 \right| } {\left|  \E^{\frac {- \bar A - \pi \I}{4m+1} + 4 \pi\I t} -1 \right| }  .
$$

\begin{sublem}
\label{sublem:FirstSumDqFarFrom0}
For every $\epsilon>0$, 
 $$
\lim_{m\to \infty}  \frac1{4m+1} \sum_{j=\left\lfloor \epsilon (4m+1) \right\rfloor}^m f_m \left( \frac j{4m+1} \right) =-2 \Im A  \int_\epsilon^\frac14 t  \cot 2\pi t \,  dt.
 $$
\end{sublem}
\begin{proof}
 For every $t$ with $0<t\leq \frac14$, 
 $$
 \log\left|  \E^{u + 4 \pi \I t} -1\right| = \log\left|  \E^{4 \pi \I t} -1\right| + \frac12 \Re u + \frac12 \cot 2\pi t \, \Im u + O\left( \left| u\right|^2 \right)
 $$
 by linear approximation. Therefore,
 $$
 f_m(t) = -2t \cot 2\pi t \, \Im A  +O\left( \frac1{m} \right)
 $$
 where the constant hidden in the symbol $O\left( \ \right)$ depends only on a compact subset of $\left]0, \frac14 \right]$ containing $t$. In particular, the function $f_m(t)$ converges to $-2t \cot 2\pi t \, \Im A $ as $m\to \infty$, and this uniformly on compact subsets of $\left]0, \frac14 \right]$. 
 
 The result then immediately follows by a Riemann sum approximation. 
\end{proof}

Based on Sublemma~\ref{sublem:FirstSumDqFarFrom0}, we now expect that 
 $$
\lim_{m\to \infty}  \frac1{4m+1} \sum_{j=1}^m f_m \left( \frac j{4m+1} \right) =-2 \Im A  \int_0^\frac14 t  \cot 2\pi t \,  dt.
 $$
 As a preliminary observation, note that this improper integral is indeed convergent as the integrand has a continuous extension to $\left[ 0, \frac14 \right]$ (this is why we split the equations (\ref{eqn:FormulaLogDq4m+1}--\ref{eqn:FormulaLogDq4m+3}) and the products of Lemma~\ref{lem:AlgebraicManipDq} the way we did). To fully justify the above limit, we need to control the terms $ f_m\left( \frac j{4m+1} \right)$ with $1\leq j < \epsilon(4m+1)$. This is based on the following sublemma. 

\begin{sublem}
\label{sublem:FirstSumDqTermsBounded}
 For a given $A$ with $\E^A \neq -1$, the absolute values of the terms 
 $$
 f_m\left( \frac j{4m+1} \right) =   (2j-1)
\log \frac
{ \left|  \E^{\frac { -A + 4\pi \I j - \pi\I} {4m+1}  } -1 \right| }
{ \left|   \E^{\frac { -\bar A + 4\pi \I j - \pi\I} {4m+1}  } -1 \right| }
 $$
 with $j \in \{1, 2, \dots, m\}$ are uniformly bounded, independently of $j$ and $m$. 
\end{sublem}

Note that, when $\Re A =0$, the function $f_m(t)$ itself is unbounded on the interval $\left[ 0, \frac14\right]$, as it has a vertical asymptotes at $t = \frac{\pi \pm  \Im A }{4\pi(4m+1)}$. 

\begin{proof}
It is easier to use a proof by contradiction. Suppose that the property fails, namely that there exists a sequence of integers $m_k$ and $j_k$, with $1\leq j_k \leq m_k$, such that
$$
\lim_{k \to \infty} \left| f_{m_k } \left( \frac{j_k}{4m_k +1 } \right)  \right| = +\infty.
$$

Passing to a subsequence if necessary, we can arrange that  the sequence $\frac{j_k}{4m_k +1 }$ has a limit $L \in \left[ 0, \frac14\right]$ as $k\to \infty$. We will then distinguish three cases, and reach a contradiction in each of them.

\medskip
\noindent\textsc{Case 1.} The limit $L=\lim_{k \to \infty} \frac{j_k}{4m_k +1 }$ is different from $0$. 

In this case, the points $ \frac{j_k}{4m_k +1 }$ stay in a compact subset of $\left]0,\frac14, \right]$. In the proof of Sublemma~\ref{sublem:FirstSumDqFarFrom0} we saw that,  as $m$ tends to $\infty$, the function $f_m(t)$ converges to $-2t \cot 2\pi t \, \Im A $ uniformly on this compact subset. Therefore, 
$$
\lim_{k \to \infty} \left| f_{m_k } \left( \frac{j_k}{4m_k +1 } \right)  \right| = \left| 2L \cot 2\pi L \, \Im A  \right| \neq + \infty,
$$
which provides the contradiction in this case. 

\medskip
\noindent\textsc{Case 2.} The integers $j_k$ are uniformly bounded. 

After passing to a subsequence, we can assume that the $j_k$ are all equal to some integer $j_0$. Then,
\begin{align*}
  f_{m_k} \left( \frac{j_k}{4m_k +1 } \right) 
  &= (2j_0-1) \log \frac{ \left|  \E^{\frac { -A + 4\pi \I j_0 - \pi\I} {4m_k+1}  } -1 \right| }
{ \left|   \E^{\frac { -\bar A + 4\pi \I j_0 - \pi\I} {4m_k+1}  } -1 \right| }
   = (2j_0-1)\log \frac {\left| \frac { -A + 4\pi \I j_0 - \pi\I} {4m_k+1} + O \left( \frac1{m_k ^2} \right) \right| } {\left|  \frac { - \bar A + 4\pi \I j_0 - \pi\I} {4m_k+1}  + O \left( \frac1{m_k^2} \right) \right|  } 
  \\
  &= (2j_0-1)\log \frac {\left|  { -A + 4\pi \I j_0 - \pi\I} + O \left( \frac1{m_k } \right) \right| } {\left|   { - \bar A + 4\pi \I j_0 - \pi\I}  + O \left( \frac1{m_k} \right) \right|  } 
  \\
  &\to  (2j_0-1)\log \frac {\left|  { -A + 4\pi \I j_0 - \pi\I} \right| } {\left|   { - \bar A + 4\pi \I j_0 - \pi\I}   \right|  }  \text{ as } k \to \infty.
\end{align*}
This limit is finite by our hypothesis that $\E^A \neq -1$, and this again contradicts the unboundedness of the sequence $ \left| f_{m_k} \left( \frac{j_k}{4m_k +1} \right)  \right|$. 

\medskip
When neither Case~1 nor Case~2 hold, we can again pass to a subsequence to be in the following remaining case. 

\medskip
\noindent\textsc{Case 3.} $\lim_{k\to \infty} \frac{j_k}{4m_k +1}=0$ and $\lim_{k\to \infty} j_k = + \infty$. 

Let us write
$$
 f_{m_k} \left( \frac {j_k}{4m_k +1} \right) =   (2j_k-1)  \Re \log \frac{  \E^{\frac { -A + 4\pi \I j_k - \pi\I} {4m_k+1}  } -1 }
{   \E^{\frac { -\bar A + 4\pi \I j_k - \pi\I} {4m_k+1}  } -1  }.
 $$
 
 Note that, as $k \to \infty$,  $\frac1{4m_k +1}$ converges to 0 faster than $\frac{j_k}{4m_k +1}$. Then, 
\begin{align*}
\frac{  \E^{\frac { -A + 4\pi \I j_k - \pi\I} {4m_k+1}  } -1 }
{   \E^{\frac { -\bar A + 4\pi \I j_k - \pi\I} {4m_k+1}  } -1  }
 &= 
 \frac
 { \E^{\frac {4\pi \I j_k} {4m_k+1}} \left( 1 - \frac {A + \pi\I}{4m_k +1} + O\left( \frac1{m_k^2} \right)\right) -1 }
{ \E^{\frac {4\pi \I j_k} {4m_k+1}} \left( 1- \frac { \bar A + \pi\I}{4m_k +1} + O\left( \frac1{m_k^2} \right)\right) -1 }
 \\
 &=
  \frac
 { \left( \E^{\frac {4\pi \I j_k} {4m_k+1}}-1 \right) -  \frac {A + \pi\I}{4m_k +1}  \E^{\frac {4\pi \I j_k} {4m_k+1}}+ O\left( \frac1{m_k^2} \right) }
{ \left( \E^{\frac {4\pi \I j_k} {4m_k+1}} -1\right)  -  \frac { \bar A + \pi\I}{4m_k +1}  \E^{\frac {4\pi \I j_k} {4m_k+1}}+ O\left( \frac1{m_k^2} \right) }
 \\
  &=
 \frac
 {1 - \frac {A + \pi\I}{4m_k +1 }  \left( 1- \E^{-\frac {4\pi \I j_k} {4m_k+1}}   \right)^{-1} + O\left( \frac1{j_k m_k} \right) }
  {1 - \frac {\bar A + \pi\I}{4m_k +1 }  \left(1-  \E^{-\frac {4\pi \I j_k} {4m_k+1}}   \right)^{-1} + O\left( \frac1{j_k m_k} \right) }
 \end{align*}
 using the fact that $\E^{ \frac{4\pi\I j_k}{4m_k +1}} -1 \asymp \frac{\pi \I  j_k}{m_k}$. Another application of the same estimate gives that
 $$
 \frac {\bar A + \pi\I}{4m_k +1 }  \left(1-  \E^{-\frac {4\pi \I j_k} {4m_k+1}}   \right)^{-1}  = O\left( \frac1{j_k} \right),
 $$
from which we conclude that
\begin{align*}
\frac{  \E^{\frac { -A + 4\pi \I j_k - \pi\I} {4m_k+1}  } -1 }
{   \E^{\frac { -\bar A + 4\pi \I j_k - \pi\I} {4m_k+1}  } -1  }
 &=
1+  \frac {\bar A - A}{4m_k +1 }  \left(1-  \E^{\frac {-4\pi \I j_k} {4m_k+1}}   \right)^{-1} + O\left( \frac1{ j_k^2} \right) 
  \\
 &= 1 - \frac{\Im A }{2\pi j_k} + O\left( \frac1{ m_k } \right)  + O\left( \frac1{ j_k^2} \right)  .
 \end{align*}
 It follows that
\begin{align*}
  f_{m_k} \left( \frac {j_k}{4m_k +1} \right) 
  &=  (2j_k-1) \Re \log \left( 1 - \frac{\Im A }{2\pi j_k} + O\left( \frac1{ m_k} \right)  + O\left( \frac1{ j_k^2} \right) \right)
  \\
  &=  - \frac{\Im A }{\pi} + O\left( \frac{j_k}{ m_k} \right)  + O\left( \frac1{ j_k} \right) 
  \\
  &\to  - \frac{\Im A }{\pi}  \text{ as } k \to \infty.
\end{align*}

This provides our final contradiction with the unboundedness of the sequence $\left|   f_{m_k} \left( \frac {j_k}{4m_k +1} \right)  \right|$, and therefore completes the proof of Sublemma~\ref{sublem:FirstSumDqTermsBounded}. 
\end{proof}

We are now ready to complete the proof of Lemma~\ref{lem:FirstSumDq}. 

Indeed, for every $\epsilon>0$, Sublemma~\ref{sublem:FirstSumDqFarFrom0} shows that 
$$
\lim_{m\to \infty}  \frac1{4m+1} \sum_{j=\left\lfloor \epsilon (4m+1) \right\rfloor}^m f_m \left( \frac j{4m+1} \right) =-2 \Im A  \int_\epsilon^\frac14 t  \cot 2\pi t \,  dt.
 $$
 while Sublemma~\ref{sublem:FirstSumDqTermsBounded} proves that 
 $$
  \frac1{4m+1} \sum_{j=1}^{\left\lfloor \epsilon (4m+1) \right\rfloor -1} f_m \left( \frac j{4m+1} \right)= O(\epsilon). 
 $$
 It follows that
  $$
\lim_{m\to \infty}  \frac1{4m+1} \sum_{j=0}^m f_m \left( \frac j{4m+1} \right) =-2 \Im A  \int_0^\frac14 t  \cot 2\pi t \,  dt 
 $$
 after observing that this improper integral converges.
 
 It turns out that the value of the above integral can be explicitly determined.

\begin{sublem}
 $$
 \int_0^\frac14 t  \cot 2\pi t \,  dt = \frac{\log 2}{8\pi}. 
 $$
\end{sublem}
\begin{proof}
 By an algebraic manipulation followed by an integration by parts, 
\begin{align*}
  \int_0^\frac14 t  \cot 2\pi t \,  dt
  &=  \int_0^\frac14  \I t + 2 \I t \frac{\E^{-4\pi\I t}}{1-\E^{-4\pi\I t} }\,  dt
  \\
 &=  \int_0^\frac14  \I t \, dt + \frac1{2\pi} \left[  t \log\left( 1- \E^{-4\pi\I t}\right) \right]_0^{\frac14} -   \frac1{2\pi} \int_0^\frac14  \log\left( 1- \E^{-4\pi\I t}\right)  \,  dt  
 \\
&= \frac\I2 \left[ t^2 \right]_0^{\frac14}+ \frac1{2\pi} \left[  t \log\left( 1- \E^{-4\pi\I t}\right) \right]_0^{\frac14} -   \frac1{8\pi^2 \I} \left[ \li \left( \E^{-4\pi \I t} \right) \right]_0^{\frac14}
\\
&= \frac{\I}{32}+ \frac{\log 2}{8\pi} - 0 +  \frac1{8\pi^2 \I}  \frac{\pi^2}{12} +  \frac1{8\pi^2 \I}  \frac{\pi^2}6 = \frac{\log 2}{8\pi}, 
\end{align*}
using the special values $\li(1)= \frac{\pi^2}6$ and $\li(-1)= -\frac{\pi^2}{12}$ of the dilogarithm function. 
\end{proof}

This computation concludes the proof of Lemma~\ref{lem:FirstSumDq} for the case when $n=4m+1$. 

The case where $n = 4m+3$ is almost identical.  
\end{proof}

\subsection{The limit of the remaining sums in (\ref{eqn:FormulaLogDq4m+1}--\ref{eqn:FormulaLogDq4m+3})}
\label{subsect:SecondSumDq}

We now tackle the last terms of the expressions (\ref{eqn:FormulaLogDq4m+1}--\ref{eqn:FormulaLogDq4m+3}) for $\left| D_n(u) \right|^{\frac1n}$. 

\begin{lem}
\label{lem:SecondSumDq}
\begin{align*}
 \lim_{m\to \infty} 
    \sum_{j=1}^m \frac{m-j+1}{4m+1}
\log\frac{ \left|  \E^{\frac { -A + 4\pi \I j - \pi\I} {4m+1}  }  -1 \right|   \left|  \E^{\frac { -\bar A + 4\pi \I j -3 \pi\I} {4m+1}  } -1\right| }
{ \left|   \E^{\frac { -\bar A + 4\pi \I j - \pi\I} {4m+1}  }  -1 \right|   \left|  \E^{\frac { -A + 4\pi \I j - 3\pi\I} {4m+1}  } -1\right|}
&= \frac14 \log \left| \frac{ \cosh \frac{A-\pi\I}4}{\cosh \frac{A+\pi\I}4} \right|
\\
 \lim_{m\to \infty} 
\sum_{j=1}^m \frac{m-j+1}{4m+3}
\log\frac{ \left|   \E^{\frac { - A + 4\pi\I j + \pi\I }{4m+3}  }   -1 \right|   \left|  \E^{\frac { - \bar A + 4\pi\I j - \pi\I }{4m+3}  }  -1\right| }
{ \left|   \E^{\frac { -\bar  A + 4\pi\I j + \pi\I }{4m+3}  }  -1 \right|  \left|  \E^{\frac { - A + 4\pi\I j - \pi\I }{4m+3}  }  -1\right| }
&= \frac14 \log\left|  \frac {(A+\pi \I)  \sinh  \frac{A-\pi \I}{4} } { (A-\pi \I)  \sinh \frac{A+\pi \I}{4} } \right|.
\end{align*}
\end{lem}

\begin{proof} As in the proof of Lemma~\ref{lem:FirstSumDq}, we focus attention on the case where $n=4m+1$. The other case is almost identical.

We first rewrite
\begin{align*}
\log\frac{ \left|  \E^{\frac { -A + 4\pi \I j - \pi\I} {4m+1}  }  -1 \right|   \left|  \E^{\frac { -\bar A + 4\pi \I j -3 \pi\I} {4m+1}  } -1\right| }
{ \left|   \E^{\frac { -\bar A + 4\pi \I j - \pi\I} {4m+1}  }  -1 \right|   \left|  \E^{\frac { -A + 4\pi \I j - 3\pi\I} {4m+1}  } -1\right|}
  &=
\Re  \log\frac{ \left(  \E^{\frac { -A + 4\pi \I j - \pi\I} {4m+1}  }  -1 \right)   \left(  \E^{\frac { - A - 4\pi \I j +3 \pi\I} {4m+1}  } -1\right) }
{ \left(   \E^{\frac { - A - 4\pi \I j + \pi\I} {4m+1}  }  -1 \right)   \left(  \E^{\frac { -A + 4\pi \I j - 3\pi\I} {4m+1}  } -1\right) }.
\end{align*}

Let us identify the leading terms of the numerator, namely
\begin{align*}
 \left(  \E^{\frac { -A + 4\pi \I j - \pi\I} {4m+1}  }  -1 \right)   \left(  \E^{\frac { - A - 4\pi \I j +3 \pi\I} {4m+1}  } -1\right) 
 &=\textstyle  \E^{\frac { -2A + 2\pi\I }{4m+1} } -2 \E^{\frac { -A + \pi\I }{4m+1} } \cos \frac{4\pi j - 2\pi}{4m+1}   +1
 \\
 &= \textstyle \left( \E^{\frac { -A + \pi\I }{4m+1} } -1 \right) ^2  +2 \E^{\frac { -A + \pi\I }{4m+1} } \left(1- \cos \frac{4\pi j-2\pi}{4m+1}  \right).
\end{align*}
We see that, as $m$ tends to $\infty$,  the second term
$$\textstyle
2 \E^{\frac { -A + \pi\I }{4m+1} } \left(1- \cos \frac{4\pi j-2\pi}{4m+1}  \right) \asymp  \left( \frac{ 4\pi j -2\pi }{4m+1}  \right)^2
$$
dominates the first term
$$
 \textstyle \left( \E^{\frac { -A + \pi\I }{4m+1} } -1 \right) ^2  \asymp \left( \frac { A - \pi\I }{4m+1}  \right)^2 .
$$

A similar property holds for the denominator. 

It is therefore natural to split
\begin{align}
 \log\frac{ \left|  \E^{\frac { -A + 4\pi \I j - \pi\I} {4m+1}  }  -1 \right|   \left|  \E^{\frac { -\bar A + 4\pi \I j -3 \pi\I} {4m+1}  } -1\right| }
{ \left|   \E^{\frac { -\bar A + 4\pi \I j - \pi\I} {4m+1}  }  -1 \right|   \left|  \E^{\frac { -A + 4\pi \I j - 3\pi\I} {4m+1}  } -1\right|}
  &=
\Re   \log \frac{ \left(  \E^{\frac { -A + 4\pi \I j - \pi\I} {4m+1}  }  -1 \right)   \left(  \E^{\frac { - A - 4\pi \I j +3 \pi\I} {4m+1}  } -1\right) } 
{2 \E^{\frac { -A + \pi\I }{4m+1} } \left(1- \cos \frac{4\pi j-2\pi}{4m+1}  \right) } \notag
\\
  &\qquad
  \label{eqn:SplitSecondTermDq4m+1}
- \Re   \log \frac{ \left(  \E^{\frac { -A - 4\pi \I j + \pi\I} {4m+1}  }  -1 \right)   \left(  \E^{\frac { - A + 4\pi \I j -3 \pi\I} {4m+1}  } -1\right) } 
{2 \E^{\frac { -A - \pi\I }{4m+1} } \left(1- \cos \frac{4\pi j-2\pi}{4m+1}  \right) } 
 \\
&\qquad\qquad + \Re   \log \frac{\E^{\frac { -A + \pi\I } {4m +1}} }{\E^{\frac { -A - \pi\I } {4m +1}} }.  \notag
\end{align}

The last term is trivial since
$$
 \Re   \log \frac{\E^{\frac { -A + \pi\I } {4m +1}} }{\E^{\frac { -A - \pi\I } {4m +1}} } =  \Re   \log  \E^ {\frac{2\pi\I}{4m +1} }   = \log \left| \E^ {\frac{2\pi\I}{4m +1} }  \right|=  0. 
$$

For the first term, our earlier computation shows that 
\begin{align*}
\Re   \log &\frac{ \left(  \E^{\frac { -A + 4\pi \I j - \pi\I} {4m+1}  }  -1 \right)   \left(  \E^{\frac { - A - 4\pi \I j +3 \pi\I} {4m+1}  } -1\right) } 
{2 \E^{\frac { -A + \pi\I }{4m+1} } \left(1- \cos \frac{4\pi j-2\pi}{4m+1}  \right) } 
  = 
\Re   \log\left(  1+  
\frac
{\left( \E^{\frac { -A + \pi\I }{4m+1} }   -1 \right)^2   } 
{2 \E^{\frac { -A + \pi\I }{4m+1} } \left(1- \cos \frac{4\pi j-2\pi}{4m+1}  \right) }  
\right)
\\
&= \Re   \log\bigg( \textstyle 1+ 
\frac12  \left( 1 + O \left( \frac1m \right) \right)  
 \left( \frac{-A+\pi\I}{4m+1} + O\left( \frac1{m^2} \right)\right)^2
  \left( \frac{2(4m+1)^2}{(4\pi j -2\pi)^2} + O\left(1 \right) \right)
\bigg)
\\
&= \textstyle \Re   \log\left(  1+  O\left( \frac1{j^2} \right) \right) = O\left( \frac1{j^2} \right) ,
\end{align*}
using the property that
$$
\frac1 {1-\cos \frac{4\pi j-2\pi}{4m+1}} = \frac{2(4m+1)^2}{(4\pi j -2\pi)^2} + O\left(1 \right)
$$
when  $1 \leq j \leq m$. (This estimate would clearly fail if $j$ was allowed to get  close to $2m$.) 

 As a consequence, if we define 
 $$
 a_m(j) = 
\begin{cases}
\frac {m-j+1}{4m+1} \Re   \log \frac{ \left(  \E^{\frac { -A + 4\pi \I j - \pi\I} {4m+1}  }  -1 \right)   \left(  \E^{\frac { - A - 4\pi \I j +3 \pi\I} {4m+1}  } -1\right) } 
{2 \E^{\frac { -A + \pi\I }{4m+1} } \left(1- \cos \frac{4\pi j-2\pi}{4m+1}  \right) } 
&\text{if } 1\leq j \leq m
\\
0 &\text{if } j > m,
\end{cases}
 $$
then the absolute value of each term $a_m(j)$ is bounded (independently of $m$) by a term $a(j)$ such that the series $\sum_{j=1}^\infty a(j)$ is convergent. In addition, the above estimates show that, for each $j$, $a_m(j)$ converges to 
$$
a_\infty(j)  = \frac14 \Re \log \left( 1 + \frac{\left( -A+\pi\I \right)^2}{16 \pi^2 \left( j-\frac12 \right)^2} \right)
$$
 as $m$ tends to $\infty$. By Lebesgue dominated convergence (or an elementary argument), it follows that 
 $$
 \lim_{m\to \infty} \sum_{j=1}^\infty a_m(j) =  \sum_{j=1}^\infty a_\infty(j) =  \sum_{j=1}^\infty \frac14 \log \left| 1+ \frac{\left( -A+\pi\I \right)^2}{16 \pi^2 \left( j-\frac12\right)^2}  \right| .
 $$
Using Euler's formula that
 $$
 \prod_{j=1}^\infty \left( 1 - \frac {z^2}{\pi^2 \left( j-\frac12\right)^2} \right) =\cos z,
 $$
 this infinite sum is equal to 
 $$
  \frac14 \log\left|  \cosh  \frac{A-\pi \I}{4}  \right| .
 $$
 
 Therefore,
 $$
 \lim_{m\to \infty} \sum_{j=1}^m  \frac{m-j+1}{4m+1}  \Re   \log \frac{ \left(  \E^{\frac { -A + 4\pi \I j - \pi\I} {4m+1}  }  -1 \right)   \left(  \E^{\frac { - A - 4\pi \I j +3 \pi\I} {4m+1}  } -1\right) } 
{2 \E^{\frac { -A + \pi\I }{4m+1} } \left(1- \cos \frac{4\pi j-2\pi}{4m+1}  \right) } =     \frac14 \log\left|  \cosh  \frac{A-\pi \I}{4}  \right|  .
 $$

A similar argument shows that
  $$
 \lim_{m\to \infty} \sum_{j=1}^m  \frac{m-j+1}{4m+1} \Re   \log \frac{ \left(  \E^{\frac { -A - 4\pi \I j + \pi\I} {4m+1}  }  -1 \right)   \left(  \E^{\frac { - A + 4\pi \I j -3 \pi\I} {4m+1}  } -1\right) } 
{2 \E^{\frac { -A - \pi\I }{4m+1} } \left(1- \cos \frac{4\pi j-2\pi}{4m+1}  \right) } =    \frac14 \log\left|  \cosh  \frac{A+\pi \I}{4}  \right|  .
 $$

Combining these two limits with (\ref{eqn:SplitSecondTermDq4m+1}), this proves the first statement of Lemma~\ref{lem:SecondSumDq}, namely that 
$$  \lim_{m\to \infty} \sum_{j=1}^m   \frac{m-j+1}{4m+1}
 \log\frac{ \left|  \E^{\frac { -A + 4\pi \I j - \pi\I} {4m+1}  }  -1 \right|   \left|  \E^{\frac { -\bar A + 4\pi \I j -3 \pi\I} {4m+1}  } -1\right| }
{ \left|   \E^{\frac { -\bar A + 4\pi \I j - \pi\I} {4m+1}  }  -1 \right|   \left|  \E^{\frac { -A + 4\pi \I j - 3\pi\I} {4m+1}  } -1\right|}
=
 \frac14 \log\left|  \frac { \cosh  \frac{A-\pi \I}{4} } {  \cosh \frac{A+\pi \I}{4} } \right| . 
$$

The second statement of Lemma~\ref{lem:SecondSumDq} is proved by a very similar argument. We will just point out the main steps. As in (\ref{eqn:SplitSecondTermDq4m+1}), we first split
\begin{align}
 \log\frac{ \left|  \E^{\frac { -A + 4\pi \I j + \pi\I} {4m+1}  }  -1 \right|   \left|  \E^{\frac { -\bar A + 4\pi \I j - \pi\I} {4m+1}  } -1\right| }
{ \left|   \E^{\frac { -\bar A + 4\pi \I j + \pi\I} {4m+1}  }  -1 \right|   \left|  \E^{\frac { -A + 4\pi \I j - \pi\I} {4m+1}  } -1\right|}
  &=
\Re   \log \frac{ \left(  \E^{\frac { -A + 4\pi \I j + \pi\I} {4m+1}  }  -1 \right)   \left(  \E^{\frac { - A - 4\pi \I j + \pi\I} {4m+1}  } -1\right) } 
{2 \E^{\frac { -A + \pi\I }{4m+1} } \left(1- \cos \frac{4\pi j}{4m+1}  \right) } \notag
\\
  &\qquad
  \label{eqn:SplitSecondTermDq4m+3}
- \Re   \log \frac{ \left(  \E^{\frac { -A - 4\pi \I j - \pi\I} {4m+1}  }  -1 \right)   \left(  \E^{\frac { - A + 4\pi \I j - \pi\I} {4m+1}  } -1\right) } 
{2 \E^{\frac { -A - \pi\I }{4m+1} } \left(1- \cos \frac{4\pi j}{4m+1}  \right) } 
 \\
&\qquad\qquad + \Re   \log \frac{\E^{\frac { -A + \pi\I } {4m +1}} }{\E^{\frac { -A - \pi\I } {4m +1}} }  \notag
\end{align}
and observe that $ \Re   \log \frac{\E^{\frac { -A + \pi\I } {4m +1}} }{\E^{\frac { -A - \pi\I } {4m +1}} }=0$.

The same estimates as before then show that the limit
 $$
 \lim_{m\to \infty} \sum_{j=1}^m  \frac{m-j+1}{4m+3} \Re     \log \frac{ \left(  \E^{\frac { -A + 4\pi \I j + \pi\I} {4m+1}  }  -1 \right)   \left(  \E^{\frac { - A - 4\pi \I j + \pi\I} {4m+1}  } -1\right) } 
{2 \E^{\frac { -A + \pi\I }{4m+1} } \left(1- \cos \frac{4\pi j}{4m+1}  \right) } 
 $$
 is equal to
 $$
  \sum_{j=1}^\infty \frac14 \log \left| 1+ \frac{\left( -A+\pi\I \right)^2}{16 \pi^2 j^2}  \right|  = \frac14 \log \left| \frac{\sinh \frac{A-\pi\I}4}{ \frac{A-\pi\I}4} \right|,
 $$
 using the other Euler formula that
 $$
 \prod_{j=1}^\infty \left( 1 - \frac {z^2}{\pi^2 j^2} \right) =\frac{\sin z}z.  
 $$
 
 Similarly,
 $$
 \lim_{m\to \infty} \sum_{j=1}^m  \frac{m-j+1}{4m+3}  \Re   \log \frac{ \left(  \E^{\frac { -A - 4\pi \I j - \pi\I} {4m+1}  }  -1 \right)   \left(  \E^{\frac { - A + 4\pi \I j - \pi\I} {4m+1}  } -1\right) } 
{2 \E^{\frac { -A - \pi\I }{4m+1} } \left(1- \cos \frac{4\pi j}{4m+1}  \right) } 
= \frac14 \log \left| \frac{\sinh \frac{A+\pi\I}4}{ \frac{A+\pi\I}4} \right|.
 $$

Combining these two limits with (\ref{eqn:SplitSecondTermDq4m+3}) then gives the expected limit 
\begin{equation*}
  \lim_{m\to \infty} 
\sum_{j=1}^m \frac{m-j+1}{4m+3}
\log\frac{ \left|   \E^{\frac { - A + 4\pi\I j + \pi\I }{4m+3}  }   -1 \right|   \left|  \E^{\frac { - \bar A + 4\pi\I j - \pi\I }{4m+3}  }  -1\right| }
{ \left|   \E^{\frac { -\bar  A + 4\pi\I j + \pi\I }{4m+3}  }  -1 \right|  \left|  \E^{\frac { - A + 4\pi\I j - \pi\I }{4m+3}  }  -1\right| }
=
 \frac14 \log\left|  \frac {(A+\pi \I)  \sinh  \frac{A-\pi \I}{4} } { (A-\pi \I)  \sinh \frac{A+\pi \I}{4} } \right|.
 \qedhere
\end{equation*}
\end{proof}

\subsection{The limit of $\left| D^q(u) \right|^{\frac1n}$}

We can now determine the limit of the terms $\left| D^q(u) \right|^{\frac1n}$ as $n$ tends to $\infty$ in a given congruence class modulo 4. 

\begin{prop}
\label{prop:AsympDq}
Let $A\in \C$ be given, with $\E^A \neq -1$. For every odd $n$, set $q = \E^{\frac{2\pi\I}n}$ and $u=q \E^{\frac1n A}$. Then,
\begin{align*}
  \lim_{\substack{n\to \infty \\ n=1\, \mathrm{mod}\, 4}}  \left| D^q(u )  \right| ^{\frac1n}
 &=2^{-\frac{\Im A}{4\pi}} \left| \frac{\cosh\frac{A-\pi\I}4 }{\cosh\frac{A+\pi\I}4  }  \right|^{\frac14}
\\
 \lim_{\substack{n\to \infty \\ n=3\, \mathrm{mod}\, 4}}  \left| D^q(u )  \right| ^{\frac1n} 
 &=2^{-\frac{\Im A}{4\pi}} \left| \frac{\sinh\frac{A-\pi\I}4 }{\sinh\frac{A+\pi\I}4  }  \right|^{\frac14} .
\end{align*}
\end{prop}
\begin{proof}
 If we let $n$ tend to $\infty$ with $n=3 \ \mathrm{mod}\ 4$, the combination of (\ref{eqn:FormulaLogDq4m+3}) with Lemmas~\ref{lem:IsolatedTermDq4m+3}, \ref{lem:FirstSumDq} and \ref{lem:SecondSumDq} shows that $\log  \left| D^q(u )  \right| ^{\frac1n} $ converges to 
 $$
 \frac14\log \frac{\left|  A - \pi\I   \right| }{\left|  \bar A - \pi\I   \right| } 
  -  \frac{\log 2}{4\pi}  \Im A 
 + \frac14 \log\left|  \frac {(A+\pi \I)  \sinh  \frac{A-\pi \I}{4} } { (A-\pi \I)  \sinh \frac{A+\pi \I}{4} } \right|
  =  -  \frac{\log 2}{4\pi}  \Im A + \frac14 \log \left| \frac{\sinh\frac{A-\pi\I}4 }{\sinh\frac{A+\pi\I}4  }  \right| .
 $$
 Therefore, $ \left| D^q(u )  \right| ^{\frac1n} $ converges to $2^{-\frac{\Im A}{4\pi}} \left| \frac{\sinh\frac{A-\pi\I}4 }{\sinh\frac{A+\pi\I}4  }  \right|^{\frac14}$. 
 
 The argument is very similar for the case where $n$ tend to $\infty$ with $n=1 \ \mathrm{mod}\ 4$.
\end{proof}

\section{The asymptotics of $\Tr \Lambda_{\phi, r}^{q_n}$}

We now have all the ingredients needed to complete our asymptotic analysis of $\Tr \Lambda_{\phi, r}^{q_n}$. 

Remember from the setup of \S \ref{sect:ComputeTrace} that we are considering  the diffeomorphism $\phi=LR=\left(
\begin{smallmatrix}
 2&1\\1&1
\end{smallmatrix}
 \right)$ of the one-puncture torus $S_{1,1}$. 
Let $[r]\in \XX(S_{1,1})$ be a $\phi$--invariant character associated to the periodic edge weight system $(a_0, b_0, c_0)$, $(a_1, b_1, c_1)$, $(a_2, b_2, c_2)=(a_0, b_0, c_0)$, and let  $\theta_v\in \C$ be such that $\E^{\theta_v}= a_0 b_0 c_0$. A choice of ``logarithms''  $A_1$, $A_2$, $V_1$, $V_2 \in \C$ as in  \S \ref{sect:ComputeTrace}  determines correction factors $\widehat l_0$, $\widehat m_0$, $\widehat n_0\in \Z$. In addition, we set $U_1 = 2\pi\I - A_1$ and $U_2 = 2\pi\I - A_2$. 

For every odd integer $n$,  let $\Lambda_{\phi, r}^q$ be the intertwiner to the data of $\phi$, $[r]$, $q=\E^{\frac{2\pi\I}n}$ and $p_v =  \E^{\frac1n \theta_v}  +\E^{-\frac1n \theta_v} $, as in Conjecture~{\upshape \ref{con:Conjecture}}. 
The following estimate immediately follows by application of Propositions~\ref{prop:AsymptoticsSn} and \ref{prop:AsympDq} to the expression of $\Tr \Lambda_{\phi, r}^{q_n}$ given in (\ref{eqn:FormulaTrace}). 

\begin{thm}
 \label{thm:MainTheorem}
For the above data, 
$$
\left| \Tr \Lambda_{\phi, r}^{q_n} \right| \asymp 
\frac{ \Big\lvert c_n(U_1, V_1, \widehat l_0) \Big\rvert \Big\lvert c_n(U_2, V_2, \widehat m_0) \Big\rvert}{ d_n(A_1) d_n(A_2)}\ 
\E^{\frac n\pi \Lambda \left( \frac\pi6 \right)}
$$
as $n$ odd tends to $\infty$, where
$$
\left| c_n(U,V, \widehat k) \right|=
\begin{cases}
 3^{-\frac14}
2^{\frac{\Im U-2\pi}{4\pi}}\ 
 \E^{ - \frac {\Re V}{6}}  
 \E^{\frac{\Re U}{12}}
&\text{ if } \widehat k \text{ is odd,}
 \\
 3^{-\frac14}
2^{\frac{\Im U-2\pi}{4\pi}}\ 
  \E^{ - \frac {2 \Re V}{3}}\ 
  \E^{\frac{\Re U}{12}}\ 
\left| 1 - \I^n \E^{ \frac 12 U }  \right|
&\text{ if } \widehat k \text{ is even},
\end{cases}
$$
and
\begin{equation*}
d_n(A)
=
\begin{cases}
2^{-\frac{\Im A}{4\pi}} \left| \frac{\cosh\frac{A-\pi\I}4 }{\cosh\frac{A+\pi\I}4  }  \right|^{\frac14}
&\text{ if } n=1\mod 4
\\
2^{-\frac{\Im A}{4\pi}} \left| \frac{\sinh\frac{A-\pi\I}4 }{\sinh\frac{A+\pi\I}4  }  \right|^{\frac14}
&\text{ if } n=3\mod 4
\end{cases}
\end{equation*}
depend only on $n$ modulo $4$. \qed
\end{thm}

\begin{cor}
\label{cor:VolConfForLR}
For the above data, 
  $$
 \lim_{n \,\text{odd} \,\to \infty} \frac1n \log \left| \Tr \Lambda_{\phi, r}^{q_n} \right| = \frac1{4\pi} \vol_\hyp (M_\phi), 
 $$
 where $ \vol_\hyp (M_\phi)$ is the volume of the complete hyperbolic metric of the mapping torus $M_\phi$. 
\end{cor}
\begin{proof}
 This is an immediate consequence of Theorem~\ref{thm:MainTheorem}, and of the fact that the hyperbolic volume $ \vol_\hyp (M_\phi)$ is in our case equal to $6\Lambda\left( \frac\pi 3\right) = 4 \Lambda\left( \frac\pi 6\right)$. 
\end{proof}

\section{Afterword}

For the diffeomorphism $\phi = LR = \left( 
\begin{smallmatrix}
 2&1\\1&1
\end{smallmatrix}
 \right)$ of the one-puncture torus $S_{1,1}$, we greatly benefitted from the phase cancellation mentioned in Remark~\ref{rem:NoPhase}, which enabled us to analyze the contributions of the  leading terms in the double sum expression of $ \Tr \Lambda_{\phi, r}^{q_n}$ given in (\ref{eqn:FormulaTrace}). 
 
 The  left-hand side of Figure~\ref{fig:LRAndLLR} is a good illustration of the cancellations in this double sum, for a specific example where $n=801$, $A_0 \approx -0.0253997 + 34.3024 \,\I$, $B_0 \approx -2.58581 - 0.229299 \,\I$, $C_0 \approx 5.33887 - 2.45979 \,\I$.  This picture, analogous to the one we encountered for a single sum in Figure~\ref{fig:SingleSum},  provides a plot of all terms in the double sum. We can here observe a nice pattern consisting of 7 ``petals'', 6 of which occur in three pairs of two petals sitting opposite to each other. It follows from our analysis in \S \ref{sect:AsymptoticsSn} that, in such a pair of opposite petals, the contributions of the two petals essentially cancel out; at least, their sum is negligible compared to the contribution of the remaining petal that has no opposite. Note that, in this example, the leading contributions to $ \Tr \Lambda_{\phi, r}^{q_n}$ come from the smallest petal. 

\begin{figure}[htbp]

\includegraphics[width=.4\textwidth]{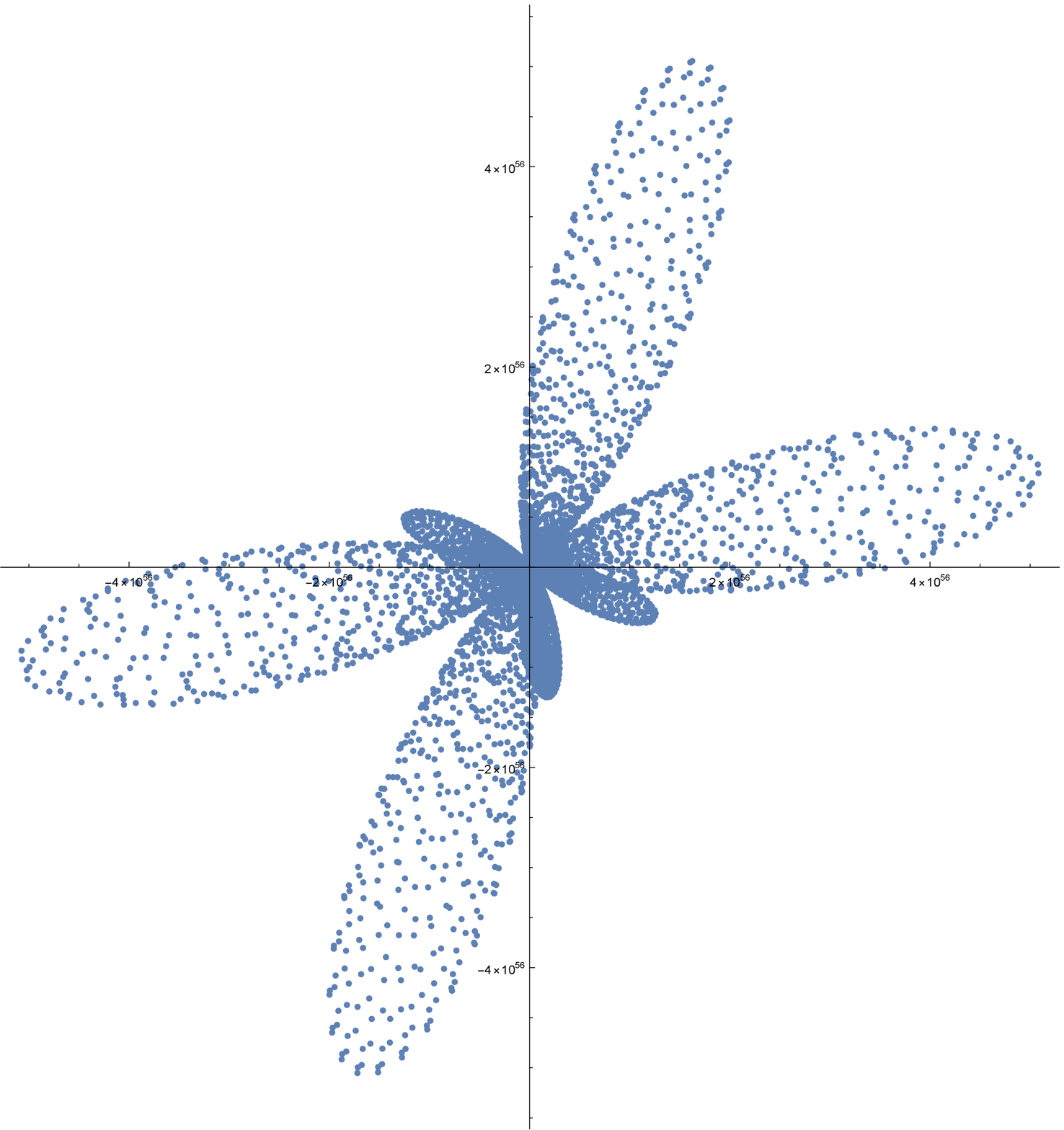}
\hskip .05\textwidth
\includegraphics[width=.4\textwidth]{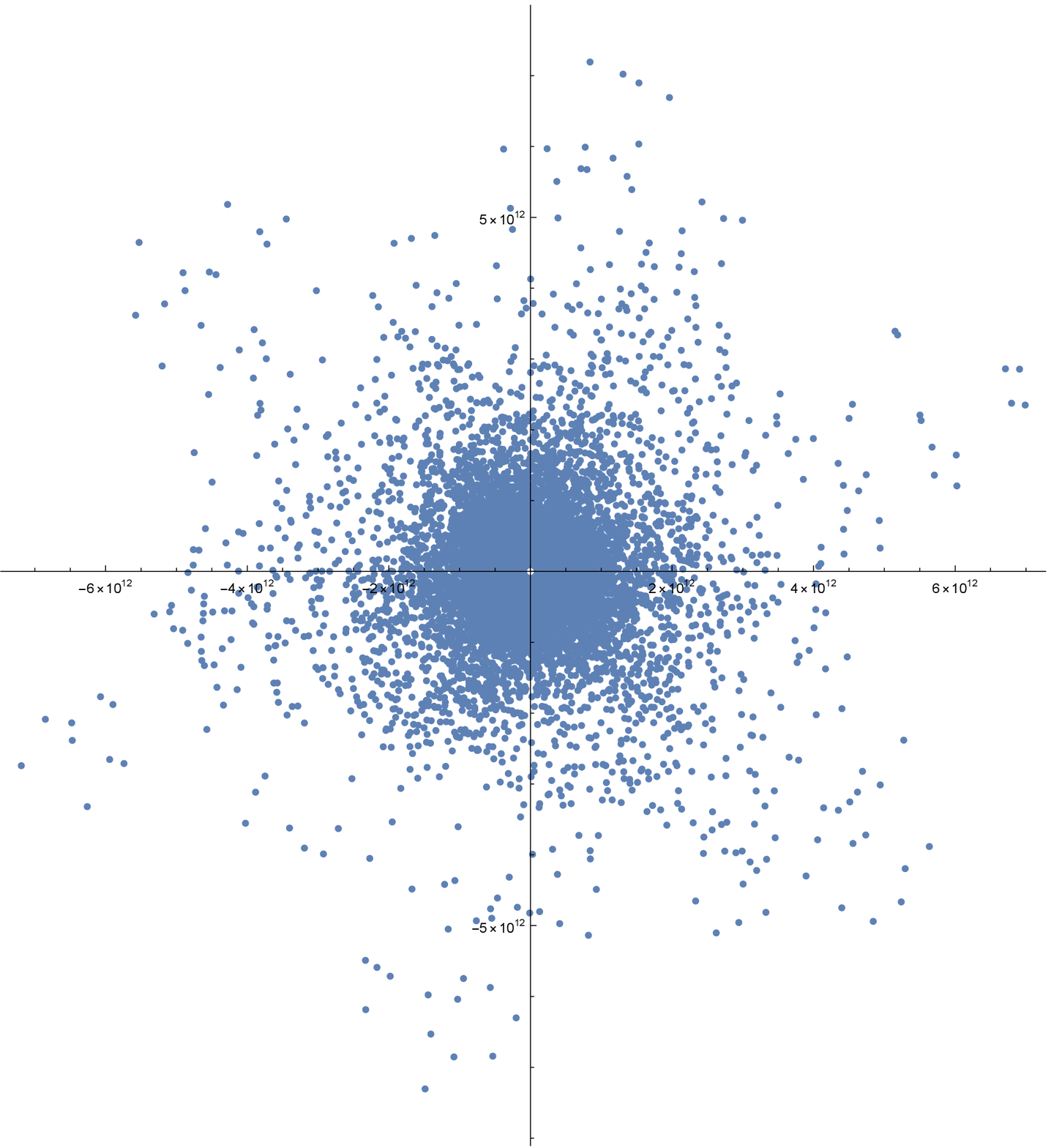}
\caption{The terms of the sums expressing $ \Tr \Lambda_{\phi, r}^{q_n}$ for explicit examples involving the diffeomorphisms $\phi = LR = \left( 
\begin{smallmatrix}
 2&1\\1&1
\end{smallmatrix}
 \right)$ and $\phi = LLR = \left( 
\begin{smallmatrix}
 3&2\\1&1
\end{smallmatrix}
 \right)$}
\label{fig:LRAndLLR}
\end{figure}

There is no nice pattern of this type for a general diffeomorphism. For instance, the right-hand side of Figure~\ref{fig:LRAndLLR} illustrates an example for the next simplest diffeomorphism $\phi = LLR = \left( 
\begin{smallmatrix}
 3&2\\1&1
\end{smallmatrix}
 \right)$, by plotting the terms of the triple sum expressing $ \Tr \Lambda_{\phi, r}^{q_n}$ in the case where $n=111$, $A_0 \approx -0.0223073 + 3.93489 \,\I$, $B_0 \approx 0.790951 + 2.38093 \,\I$, $C_0 \approx -0.42207 + 0.752766 \,\I$. It is clearly hard to see any logic in this picture. In fact,  the  elementary techniques that we used in the current paper are here ineffective, and we will need the more sophisticated harmonic analysis and differential topology methods of \cite{BWY3} to introduce some order in the chaos. Interestingly enough, this analysis will lead us to an outcome that is similar to the one we observed for $LR$, with a grouping of leading terms into blocks that, either cancel out in pairs, or add up to provide the leading contributions.

\bibliographystyle{amsalpha}
\bibliography{BWY2}

\providecommand{\bysame}{\leavevmode\hbox to3em{\hrulefill}\thinspace}
\providecommand{\MR}{\relax\ifhmode\unskip\space\fi MR }
\providecommand{\MRhref}[2]{%
  \href{http://www.ams.org/mathscinet-getitem?mr=#1}{#2}
}
\providecommand{\href}[2]{#2}
\begin{thebibliography}{FKBL19}

\bibitem[Bon09]{BonBook}
Francis Bonahon, \emph{Low-dimensional geometry: from euclidean surfaces to
  hyperbolic knots}, Student Mathematical Library, vol.~49, American
  Mathematical Society, Providence, RI; Institute for Advanced Study (IAS),
  Princeton, NJ, 2009, IAS/Park City Mathematical Subseries.

\bibitem[BW16]{BonWon3}
Francis Bonahon and Helen Wong, \emph{Representations of the {K}auffman bracket
  skein algebra {I}: invariants and miraculous cancellations}, Invent. Math.
  \textbf{204} (2016), no.~1, 195--243.

\bibitem[BWY21]{BWY1}
Francis Bonahon, Helen Wong, and Tian Yang, \emph{Asymptotics of quantum
  invariants of surface diffeomorphisms {I}: conjecture and algebraic
  computations}, \texttt{arXiv:2112.12852}, 2021.

\bibitem[BWY22]{BWY3}
\bysame, \emph{Asymptotics of quantum invariants of surface diffeomorphisms
  {III}: the one-puncture torus}, in preparation, 2022.

\bibitem[Fad95]{Fadd}
Lyudvig~D. Faddeev, \emph{Discrete {H}eisenberg-{W}eyl group and modular
  group}, Lett. Math. Phys. \textbf{34} (1995), no.~3, 249--254. \MR{1345554}

\bibitem[FC99]{CheFoc1}
Vladimir~V. Fock and Leonid~O. Chekhov, \emph{Quantum {T}eichm\"{u}ller
  spaces}, Teoret. Mat. Fiz. \textbf{120} (1999), no.~3, 511--528.

\bibitem[FK94]{FadKash}
Lyudvig~D. Faddeev and Rinat~M. Kashaev, \emph{Quantum dilogarithm}, Modern
  Phys. Lett. A \textbf{9} (1994), no.~5, 427--434.

\bibitem[FKBL19]{FroKanLe1}
Charles Frohman, Joanna Kania-Bartoszy\'nska, and Thang L\^{e}, \emph{Unicity
  for representations of the {K}auffman bracket skein algebra}, Invent. Math.
  \textbf{215} (2019), no.~2, 609--650.

\bibitem[GJS19]{GanJorSaf}
Iordan Ganev, David Jordan, and Pavel Safronov, \emph{The quantum {F}robenius
  for character varieties and multiplicative quiver varieties}, preprint,
  \texttt{arXiv:1901.11450}.

\bibitem[WY20]{WongYang1}
Ka~Ho Wong and Tian Yang, \emph{On the {V}olume {C}onjecture for hyperbolic
  {D}ehn-filled $3$-manifolds along the figure-eight knot}, preprint,
  \texttt{arXiv:2003.10053}, 2020.

\end{thebibliography}
 
\end{document}